\documentclass[10pt]{amsart}

\usepackage{amsmath,amsthm,amssymb}
\usepackage[mathscr]{eucal}
\usepackage{graphicx,array,pinlabel,enumerate,enumitem}
\usepackage[small,bf,margin=24pt]{caption}

\theoremstyle{plain}
\newtheorem{thm}{Theorem}[section]
\newtheorem{prop}[thm]{Proposition}
\newtheorem{lemma}[thm]{Lemma}
\newtheorem{cor}[thm]{Corollary}

\theoremstyle{definition}\newtheorem{Def}[thm]{Definition}
\theoremstyle{remark}\newtheorem{rmk}[thm]{Remark}
\newtheorem{ex}[thm]{Example}

\numberwithin{equation}{section}

\newcommand*{\N}{\ensuremath{\mathbf N}}
\newcommand*{\Z}{\ensuremath{\mathbf Z}}
\newcommand*{\Q}{\ensuremath{\mathbf Q}}
\newcommand*{\R}{\ensuremath{\mathbf R}}
\newcommand*{\C}{\ensuremath{\mathbf C}}
\newcommand*{\spin}{\ensuremath{\mathrm{Spin}^c}}

\DeclareMathOperator{\bip}{Bip}
\DeclareMathOperator{\conv}{Conv}
\DeclareMathOperator{\id}{id}

\begin{document}

\title[Root polytopes, Tutte polynomials, and duality]{Root polytopes, Tutte polynomials, and a duality theorem for bipartite graphs}

\author{Tam\'as K\'alm\'an}

\address{Tokyo Institute of Technology}
\email{kalman@math.titech.ac.jp}
\urladdr{www.math.titech.ac.jp/\char126 kalman}

\author{Alexander Postnikov}

\address{Massachusetts Institute of Technology}
\email{apost@math.mit.edu}

\date{}


\thanks{The first author was supported by consecutive Japan Society for the Promotion of Science (JSPS) Grant-in-Aids for Young Scientists B (nos.\ 21740041 and 25800037).}

\begin{abstract}
Let $G$ be a connected bipartite graph with color classes $E$ and $V$ and root polytope $Q$. Regarding the hypergraph $\mathscr H=(V,E)$ induced by $G$, we prove that the interior polynomial of $\mathscr H$ is equivalent to the Ehrhart polynomial of $Q$, which in turn is equivalent to the $h$-vector of any triangulation of $Q$. It follows that the interior polynomials of $\mathscr H$ and its transpose $\overline{\mathscr H}=(E,V)$ agree.

When $G$ is a complete bipartite graph, our result recovers a well known hypergeometric identity due to Saalsch\"utz.
It also implies that certain extremal coefficients in the Homfly polynomial of a special alternating link can be read off of an associated Floer homology group. 
\end{abstract}

\maketitle

\section{Introduction}

The interior polynomial $I$ \cite{hypertutte} is an invariant of hypergraphs (and of integer polymatroids) that generalizes the specialization $T(x,1)$ of the Tutte polynomial of ordinary graphs (matroids). It first arose as a byproduct of the first author's study of polynomial invariants of knots. In that context, it was natural to conjecture that $I_{\mathscr H}=I_{\overline{\mathscr H}}$, where $\overline{\mathscr H}$ is the `abstract dual' (also known as transpose) of the hypergraph $\mathscr H$ resulting from interchanging the roles of its vertices and hyperedges. That is, $I$ is an invariant of the bipartite graph $G$ that captures the common structure of $\mathscr H$ and $\overline{\mathscr H}$.

In this paper we verify that conjecture. As an intermediary between the two polynomials, we insert the Ehrhart polynomial $\varepsilon_G$ of the root polytope~$Q_G$, which was introduced and studied by the second author \cite{alex}. (For the definition of $Q_G$ see Section \ref{sec:root}.) We also show that $\varepsilon_G$ has an equivalent description as the $h$-vector 
of an arbitrary triangulation of $Q_G$. (In particular, all triangulations of $Q_G$ have the same $h$-vector $h_G$.) Explicitly, we will prove the following.

\begin{thm}\label{thm:main}
If $E$ and $V$ are the color classes of the connected bipartite graph~$G$, then the interior polynomial~$I$ of either hypergraph induced by $G$ is related to $h_G$ via the formula
\begin{equation}\label{eq:I=h}
I
(x)=
x^{|E|+|V|-1}h_G(x^{-1}).
\end{equation}
\end{thm}


We chose the somewhat unusual labels for the color classes to ease the transition to the hypergraph point of view.
In the proof, we argue that if the sequence $a_0,a_1,\ldots,a_{|E|+|V|-2}$ of rational numbers (which will later turn out to consist of non-negative integers, with at least $\max\{\,|E|,|V|\,\}-1$ zeros at the end) is such that for any non-negative integer $s$, we have\footnote{The dimension of $Q_G$ is $|E|+|V|-2$. The binomial coefficient on the right hand side is the number of lattice points in a standard simplex of that dimension and of sidelength $s-k$. When considered for all $k$, these binomial coefficients form a basis 
in the sense of Lemma \ref{lem:basis}.}
\begin{equation}\label{eq:ehrhart}
\varepsilon_G(s):=\left|(s\cdot Q_G)\cap(\Z^{E}\oplus\Z^{V})\right|=\sum_{k=0}^{|E|+|V|-2}a_k{s+|E|+|V|-2-k \choose |E|+|V|-2},
\end{equation}
then both sides of \eqref{eq:I=h} agree with $\sum_{k=0}^{|E|+|V|-2}a_kx^k$. This boils down to a delicate analysis of interior faces in a triangulation of $Q_G$ in which we will rely on previous work by the second author \cite{alex}. We note that a key result of \cite{alex} is the claim $I_{\mathscr H}(1)=I_{\overline{\mathscr H}}(1)$, which is proven by equating both sides to $h_G(1)$. That is, in \cite{alex} it is shown that $\mathscr H$ and $\overline{\mathscr H}$ have the same number of so-called hypertrees (see Definition \ref{def:hiperfa}) and that that number is the number of simplices in each triangulation of $Q_G$.

If $G$ is a complete bipartite graph on $(m+1)+(n+1)$ vertices then $Q_G=\Delta_m\times\Delta_n$ is the product of an $m$- and an $n$-dimensional unit simplex.
If we consider \eqref{eq:ehrhart} in this special case and for $a_k$ we substitute the coefficients of the interior polynomial from a separate 
computation \cite[Example 7.2]{hypertutte}, we obtain the well known identity
\begin{equation}\label{saalschutz}
{s+m \choose m}{s+n \choose n}=
\sum_{k=0}^{\min\{m,n\}}{m \choose k}{\,n\, \choose k}{s+m+n-k \choose m+n}
\end{equation}
due to Saalsch\"utz
. See Example \ref{saal}. The other form in which \eqref{saalschutz} often appears,
\begin{equation}\label{otherform}
{q \choose m}{q \choose n}=\sum_{k=0}^{\min\{m,n\}}{m \choose k}{\,n\, \choose k}{q+k \choose m+n},
\end{equation}
is related to \eqref{saalschutz} by an application of Ehrhart reciprocity to $Q_G$.

Both our main theorem and its proof are entirely combinatorial. Yet the result serves as a crucial step in a knot theoretical program advanced by Juh\'asz, Mura\-kami, Rasmussen, and the first author. Namely, in \cite{homfly} it is shown that certain extremal coefficients in the Homfly polynomial of a special alternating link agree with coefficients in the $h$-vector $h_G$, where $G$ is the Seifert graph of the link. By the main result of this paper, we see that those same Homfly coefficients are also the coefficients of the interior polynomial of either one of the corresponding two hypergraphs. By definition, interior polynomials are derived from the structure of the set of hypertrees of the hypergraph (see \cite{hypertutte} and Section~\ref{sec:interior}). Finally, in \cite{jkr} it is proven that the hypertrees in question appear as the supporting $\spin$ structures of some Floer homology groups that are naturally associated to the special alternating link. Hence it is possible to read certain Homfly coefficients directly out of Floer homology. To the best of our knowledge such a result is the first of its kind. In order to keep this paper concise, we omit a full explanation and mention only one topological corollary of Theorem \ref{thm:main}.

\begin{cor}
Regarding a positive special alternating link diagram with Seifert graph $G$, the coefficient of $z^{b_1(G)}$ in its Homfly polynomial $P(v,z)$ is $v^{b_1(G)}I(v^2)$, where $I$ is the common interior polynomial of the two hypergraphs induced by $G$.
(Here the first Betti number (also known as nullity) $b_1(G)$ is the highest exponent of $z$ that occurs in $P$ \cite{morton}.)
\end{cor}

Observations by Mura\-kami and the first author \cite{homfly} concerning generalized parking functions \cite{ps} for the dual of a plane bipartite graph $G$ yield that their natural enumerator also coincides with the interior polynomial of $G$. See Corollary \ref{cor:park}.

Some of our results provide new information on the Tutte polynomial $T(x,y)$ in the classical (graph and matroid) context. As a special case of Theorem \ref{thm:mink} we see that $T(1,y)$ and $T(x,1)$ are equivalent to lattice point counts in the Minkowski sum of a simplex or inverted simplex of variable sidelength with the spanning tree polytope (for graphs) or base polytope (for matroids). Our duality theorem (Corollary~\ref{cor:duality}) can be used to express $T(x,1)$ in the case of a graph 
as a sum written in terms of activities associated to vertices instead of edges, cf.\ Corollary \ref{cor:reliable}.

This paper updates its predecessor \cite{hypertutte} by replacing the two longest proofs. Namely, Corollary \ref{cor:duality} implies that the two deletion-contraction formulas established there \cite[Proposition 6.14 and Theorem 7.3]{hypertutte} are equivalent, so a separate proof of the latter is no longer necessary. We also give a new proof of the well-definedness of the interior (and exterior) polynomial, shorter and more elegant than the original \cite[Theorem 5.4]{hypertutte}.

The paper is organized as follows. In Section \ref{sec:interior} we recall the construction of the interior polynomial and re-prove
that it is well-defined. In Section \ref{sec:root} we summarize some facts about the root polytope and establish the equivalence of its Ehrhart polynomial with the common $h$-vector of its triangulations. Section \ref{sec:prep} contains a few lemmas needed for the proof of Theorem \ref{thm:main}, which is given in Section \ref{sec:main}. At the end of the paper we discuss several corollaries.

Acknowledgements. We are grateful to P\'eter Frenkel for recognizing Saalsch\"utz's identity, and to Suho Oh and Dylan Thurston for stimulating conversations.

\section{The interior polynomial}\label{sec:interior}

In this section we set notation, recall some necessary material from \cite{hypertutte}, and give a new, shorter proof of one of that paper's key results.

\subsection{Hypertrees}

A \emph{bipartite graph} is a triple $G=(V_0,V_1,E)$, where $V_0$ and $V_1$ are disjoint finite sets, called \emph{color classes}, and $E$ is a finite set of edges connecting an element of $V_0$ to an element of $V_1$. Multiple edges are not allowed. We will treat $(V_0,V_1,E)$, $(V_1,V_0,E)$, and the graph $(V_0\cup V_1,E)$ as the same object.

A \emph{hypergraph} is a pair $\mathscr H=(V,E)$, where $V$ is a finite set and $E$ is a finite multiset of non-empty subsets of $V$. Elements of $V$ are called \emph{vertices} and the elements of $E$ are the \emph{hyperedges}. 

The sets $V$ and $E$ that constitute the hypergraph~$\mathscr H$ may be viewed as the color classes of a bipartite graph $\bip\mathscr H$, where we connect $v\in V$ to $e\in E$ with an edge if $v\in e$. We will call the result the \emph{bipartite graph associated to the hypergraph}~$\mathscr H$.
We will sometimes refer to $E$ as the emerald color class of $\bip(V,E)$ and to $V$ as the violet color class.

The construction of $\bip\mathscr H$ is reversible if we specify one color class in the bipartite graph $G=(V_0,V_1,E)$. Let us denote the resulting hypergraphs with
\begin{equation}\label{eq:absdual}
\mathscr H_0=(V_1,V_0)\quad\text{and}\quad\mathscr H_1=(V_0,V_1).
\end{equation}

\begin{Def}
The bipartite graph $G$ above is said to \emph{induce} the hypergraphs~$\mathscr H_0$ and $\mathscr H_1$. Two hypergraphs are called \emph{abstract duals} if they can be obtained in the form \eqref{eq:absdual}. In other words, the abstract dual (also known as the \emph{transpose}) $\overline{\mathscr H}=(E,V)$ of a hypergraph $\mathscr H=(V,E)$ is defined by interchanging the roles of its vertices and hyperedges.
\end{Def}

A certain generalization of the notion of a spanning tree to hypergraphs plays a central role in this paper. It appears in both authors' previous work. Although it was first introduced as the `left (or right) degree vector' \cite{alex}, we will use the following terminology instead.

\begin{Def}\label{def:hiperfa}
Let $\mathscr H=(V,E)$ be a hypergraph so that its associated bipartite graph $\bip\mathscr H$ is connected. (In this situation it is common to call $\mathscr H$ itself connected.) By a \emph{hypertree} in $\mathscr H$ we mean a function (vector)~$\mathbf f\colon E\to\N=\{\,0,1,\ldots\,\}$ so that a spanning tree of $\bip\mathscr H$ can be found which has valence $\mathbf f(e)+1$ at each $e\in E$. Such a spanning tree is said to \emph{realize} or to \emph{induce} $\mathbf f$. We denote the set of all hypertrees in $\mathscr H$ with $B_{\mathscr H}$.
\end{Def}

It is easy to show that all hypertrees $\mathbf f$ in $\mathscr H$ satisfy $\sum_{e\in E}\mathbf f(e)=|V|-1$. The set $B_{\mathscr H}$ is such that $(\conv B_{\mathscr H})\cap\Z^E=B_{\mathscr H}$, cf.\ Lemma \ref{lem:ineq}, where $\conv$ denotes the usual convex hull. We will call $\conv B_{\mathscr H}$ the \emph{hypertree polytope} of $\mathscr H$.

The definition of the interior polynomial is based on hypertrees and the following concept. It is a natural extension of internal activity used in the case of graphs (and 
matroids). A small difference is that if we specialize our version to graphs, external edges of a spanning tree become internally active; another is that we count inactive hyperedges instead of active ones. But since the number of external edges is the same for all spanning trees (namely, the first Betti number of the graph, also known as its nullity), all this just mirrors and shifts the distribution of the `classical' statistic. 

\begin{Def}\label{def:activity}
Let $(V,E)$ be a hypergraph and let us order the set $E$ arbitrarily. A hyperedge $e\in E$ is \emph{internally active} with respect to the hypertree $\mathbf f$ if it is not possible to decrease $\mathbf f(e)$ by $1$ and increase $\mathbf f$ at a hyperedge smaller than $e$ by $1$ so that another hypertree results. Let $\iota(\mathbf f)$ denote the number of internally active hyperedges with respect to $\mathbf f$. 
\end{Def}

If the hypertree $\mathbf f\in B_{\mathscr H}$ and the hyperedges $e, e'\in E$ are such that changing the value $\mathbf f(e)$ to $\mathbf f(e)-1$ and the value $\mathbf f(e')$ to $\mathbf f(e')+1$ results in another hypertree $\mathbf f'$, then we say that $\mathbf f$ and $\mathbf f'$ are related by a \emph{transfer of valence} from $e$ to $e'$.

We call a hyperedge \emph{internally inactive} with respect to a hypertree if it is not internally active and denote the number of such hyperedges (for a given $\mathbf f$) by $\bar\iota(\mathbf f)=|E|-\iota(\mathbf f)$. This value will be called the \emph{internal inactivity} of $\mathbf f$.

\begin{Def}\label{def:I}
Let $\mathscr H=(V,E)$ be a hypergraph so that $\bip\mathscr H$ is connected. For some fixed linear order on $E$ we consider the generating function of internal inactivity, $I_\mathscr H(\xi)=\sum_{\mathbf f\in B_{\mathscr H}} \xi^{\bar\iota(\mathbf f)}$, and call it the \emph{interior polynomial} of $\mathscr H$.
\end{Def}

If $\mathscr H$ is a graph with Tutte polynomial $T(x,y)$, then its interior polynomial is 
$\xi^{|V|-1}T(1/\xi,1)$.
For any $\mathscr H$, the polynomial $I_{\mathscr H}$ is independent of the order that is used to define $\bar\iota$ \cite[Theorem 5.4]{hypertutte}. We give a new proof of this fact in the next subsection. There is an analogous notion of external activity and a corresponding exterior polynomial of hypergraphs \cite{hypertutte}. 
However it does not play much of a role in this paper other than in the planar case, cf.\ \cite[Theorem 8.3]{hypertutte} and Corollary \ref{cor:polynomials}. 

\begin{ex}\label{ex:k23poly}
Let the complete bipartite graph $K_{2,3}$ induce the hypergraph $\mathscr H$ with three hyperedges and $\overline{\mathscr H}$ with two hyperedges. Both have three hypertrees, namely $(1,0,0),(0,1,0),(0,0,1)$ in $\mathscr H$ and $(2,0),(1,1),(0,2)$ in $\overline{\mathscr H}$. In both cases, up to isomorphism there is only one way to order hyperedges and the resulting interior polynomial is $1+2\xi^2$.
\end{ex}

\begin{ex}\label{ex:12a1097}
Both hypergraphs induced by the plane bipartite graph pictured in Figure \ref{fig:12a1097} have the interior polynomial $I(\xi)=1+4\xi+7\xi^2+4\xi^3$. 
This claim can be verified by a completely elementary computation. For example, if we treat the four emerald points $a,b,c,d$ as hyperedges and write hypertrees $\mathbf f$ in the form of a vector $(\mathbf f(a),\mathbf f(b),\mathbf f(c),\mathbf f(d))$, then $(1,0,1,2)$ becomes a hypertree. This is because the spanning tree pictured on the right realizes it. 

\begin{figure}[h!] 
\labellist
\small
\pinlabel $a$ at 300 -30
\pinlabel $b$ at 560 230
\pinlabel $c$ at 230 560
\pinlabel $d$ at 240 350
\endlabellist
   \centering
   \includegraphics[width=2.5in]{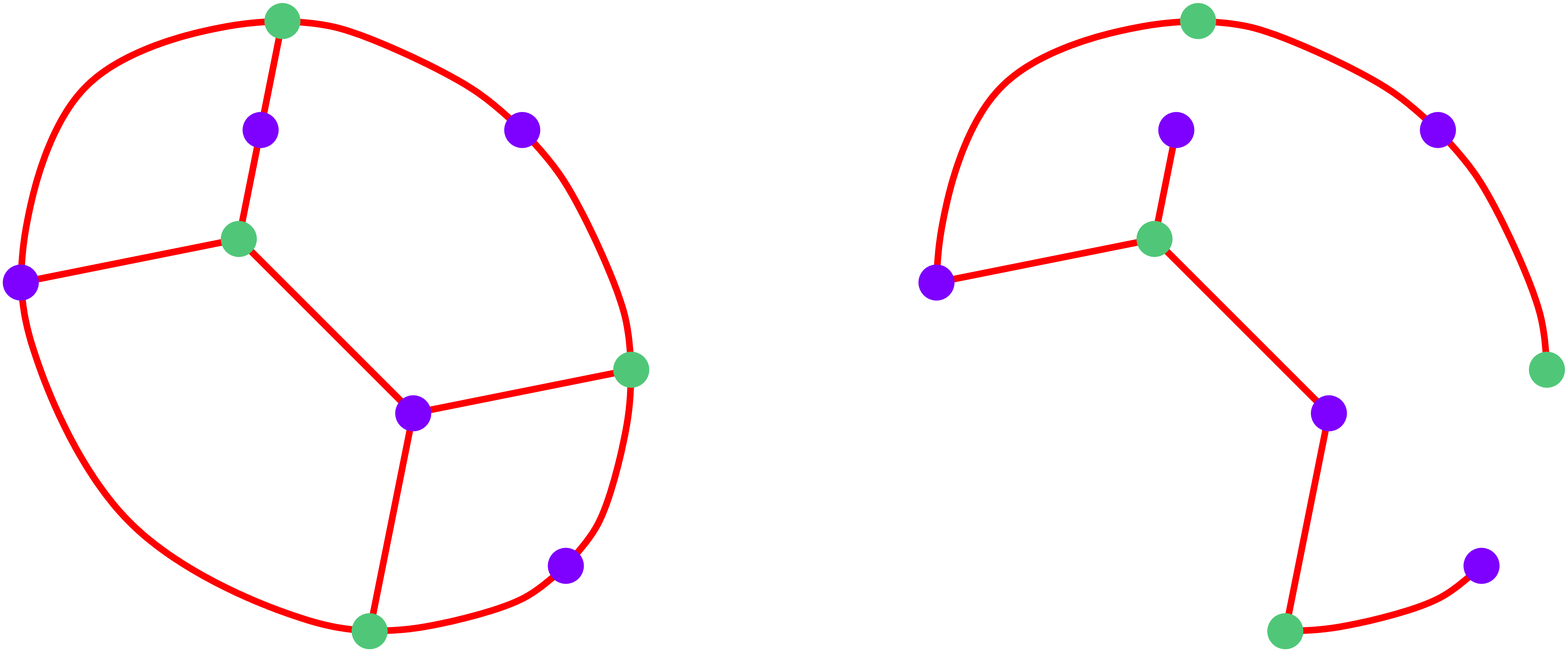} 
   \caption{A plane bipartite graph with one of its spanning trees.}
   \label{fig:12a1097}
\end{figure}

The other $15$ hypertrees (and the $16$ hypertrees of the abstract dual hypergraph with the five violet hyperedges) may be found by solving the system of inequalities given in Lemma \ref{lem:ineq} below, whose necessity is particularly easy to see, or by any number of ad hoc methods\footnote{Another systematic way is to construct the Minkowski sum $P_E$ (see \cite{alex} and Section \ref{sec:root}) and enumerate integer translates of the unit simplex in it, cf.\ the discussion right before and after formula \eqref{eq:bigcell}. If $|E|=4$ then $P_E$ is $3$-dimensional. But for instance for our abstract dual hypergraph of five hyper\-edges, it becomes $4$-dimensional so that the count is much harder to carry out `by hand.'}. We do not, however, recommend enumerating all spanning trees as there are $217$ of those.

For any order of the hyperedges, the smallest hyperedge is always internally active with respect to any hypertree. If a hypertree assigns $0$ to a hyperedge, then that hyperedge is automatically internally active with respect to the hypertree. So if we use the order $a<b<c<d$ in our example, then with respect to $(1,0,1,2)$ the hyperedges $a$ and $b$ are internally active. On the other hand, $c$ and $d$ are internally inactive. This is demonstrated by the upper right and lower left spanning trees (and their induced hypertrees) of Figure \ref{fig:badcase}. Indeed, one shows a transfer of valence from $c$ to $b$ and the other a transfer of valence from $d$ to $a$. Hence $\bar\iota(1,0,1,2)=2$ so that the hypertree contributes $1$ to the coefficient of $\xi^2$ in $I$. The other $15$ (or rather, $31$) internal inactivities can be computed the same way.
\end{ex}

At times it will be convenient to rely on submodular function techniques, that is to say, on Lemma \ref{lem:tight} below. This is made possible by the following observations.

For a bipartite graph $G$ with color classes $E$ and $V$ and for a subset $E'\subset E$, we let $G\big|_{E'}$ denote the graph formed by $E'$, all edges of $G$ adjacent to elements of $E'$, and their endpoints in $V$. We let $c(E')$ denote the number of connected components of $G\big|_{E'}$ and we also let $\bigcup E'=V\cap \left(G\big|_{E'}\right)$ (this notation is natural if we view $(V,E)$ as a hypergraph). Finally we let $\mu(\varnothing)=0$ and otherwise
\[\mu(E')=\left|\textstyle{\bigcup E'}\right|-c(E').\]
Then $\mu$ is a non-decreasing (i.e., $E''\subset E'$ implies $\mu(E'')\le\mu(E')$) submodular function on the power set of $E$. The latter means that for all $A,B\subset E$ we have
\[\mu(A)+\mu(B)\ge\mu(A\cup B)+\mu(A\cap B).\]
This fact is probably `part of the folklore.' 
For a proof, see \cite[Proposition 4.7]{hypertutte}.


\begin{lemma}[\cite{hypertutte}, Theorem 3.4]\label{lem:ineq}
If $G$ is connected, then $\mu(E)=|V|-1$ and the hypertrees $\mathbf f$ in $(V,E)$ are exactly the integer solutions of the system of inequalities
\[\mathbf f(e)\ge0\text{ for all }e\in E; \quad \sum_{e\in E}\mathbf f(e)=|V|-1; \quad \sum_{e\in E'}\mathbf f(e)\le\mu(E')\text{ for all }E'\subset E.\]
\end{lemma}

We note that the non-negativity of $\mathbf f$ follows from the other constraints.
We say that the set $E'\subset E$ is \emph{tight} at $\mathbf f$ if $\sum_{e\in E'}\mathbf f(e)=\mu(E')$ holds. If $E'$ is tight at $\mathbf f$ and it so happens that for another hypertree $\mathbf g$, we have $\sum_{e\in E'}\mathbf f(e)=\sum_{e\in E'}\mathbf g(e)$ (for example, if $\mathbf f$ and $\mathbf g$ differ by a transfer of valence between elements of $E'$), then $E'$ is also tight at $\mathbf g$. The next lemma follows immediately from \cite[Theorem 44.2]{sch}.

\begin{lemma}\label{lem:tight}
If the sets $A,B\subset E$ are both tight at the hypertree $\mathbf f$, then so are $A\cup B$ and $A\cap B$.
\end{lemma}

Let now $\Gamma$ be a spanning tree in $G$ which induces the hypertree $\mathbf f\colon E\to\N$. For any subset $E'\subset E$, the connected components of $\Gamma\big|_{E'}$ induce a partition of $E'$.

\begin{lemma}\label{lem:tightcomp}
Let $E'\subset E$ be a tight set at the hypertree $\mathbf f$ and $\Gamma\subset G$ a realization of $\mathbf f$. Then each part in the partition of $E'$ induced by $\Gamma$ is itself tight at $\mathbf f$.
\end{lemma}

\begin{proof}
An equivalent definition of $\mu$ is that a cycle-free subgraph of $G\big|_{E'}$ is a spanning forest of $G\big|_{E'}$ if and only if it has $|E'|+\mu(E')$ edges. Since now $E'$ is tight, we have $|E'|+\mu(E')=|E'|+\sum_{e\in E'}\mathbf f(e)=\sum_{e\in E'}(\mathbf f(e)+1)$, which is the number of edges in $\Gamma\big|_{E'}$ (note that $\mathbf f(e)+1$ is the degree of $e$ in $\Gamma$ and each edge of $\Gamma\big|_{E'}$ has a unique endpoint in $E'$). Thus $\Gamma\big|_{E'}$ is a spanning forest in $G\big|_{E'}$; in particular, the connected components of $\Gamma\big|_{E'}$ and those of $G\big|_{E'}$ induce the same partition of $E'$ and for each part $E''$ of that partition, $\Gamma\big|_{E''}$ is a spanning tree in $G\big|_{E''}$. By the same logic as above the latter implies $\sum_{e\in E''}(\mathbf f(e)+1)=|E''|+\mu(E'')$, i.e., $\sum_{e\in E''}\mathbf f(e)=\mu(E'')$, as claimed.
\end{proof}

\subsection{Order-independence}

The fact that the interior polynomial is independent of the order imposed on the hyperedges was proved in \cite{hypertutte} using a straightforward but rather long argument. The reasoning that establishes the main theorem of this paper can also be viewed as a rather circuitous proof of order-independence. Here we give a third argument which is shorter than the other two. It also serves to emphasize that interior (and exterior) polynomials correspond to Ehrhart-type lattice point counts not just in the sense of \eqref{eq:ehrhart} but in a more direct way as well. (One might expect this observation to be the basis of a short proof of Theorem \ref{thm:main}, but so far no such connection has been found.)

We fix the connected hypergraph $\mathscr H=(V,E)$ whose set of hypertrees is $B_{\mathscr H}$. 
Let the \emph{inverted standard simplex} $\nabla_E$ be the convex hull of the set $\{\,-\mathbf i_{\{e\}}\in\R^E\mid e\in E\,\}$. Here and in the rest of the paper $\mathbf i_S$ denotes the indicator function of the set $S\subset E$. For a positive 
integer $k$ we are going to count lattice points in the Minkowski sum 
\[\conv B_{\mathscr H}+k\nabla_E.\]
First of all, each of these is of the form $\mathbf f+\mathbf v$, 
where $\mathbf f$ is a hypertree and $\mathbf v$ is an integer vector in $k\nabla_E$, i.e., a vector with non-positive integer entries whose sum is $-k$. 
This follows from \cite[Corollary 46.2c]{sch} because $\conv B_{\mathscr H}$ is the base polytope of an integer polymatroid (by Lemma \ref{lem:ineq} and the submodularity of $\mu$), and $\nabla_E$ is a translate of another such polytope by the integer vector $-\mathbf i_E$: indeed $\nabla_E$ corresponds to the $(|E|-1)$-uniform matroid on the set $E$. 

After fixing an order on $E$, we partition our lattice points $\mathbf f+\mathbf v$ according to the lowest possible $\mathbf f\in B_{\mathscr H}$ in the following version of the lexicographic order: 
\begin{align}
\label{eq:order}
&\mathbf f_1\text{ is smaller than }\mathbf f_2\text{ if they are different and if for the smallest}\\
\nonumber
&\text{element }e\text{ of }E\text{ with }\mathbf f_1(e)\ne\mathbf f_2(e)\text{, we have }\mathbf f_1(e)>\mathbf f_2(e)
\end{align}
(this is rather natural because the sum of the entries in $\mathbf f$ is fixed). For a hypertree~$\mathbf f$, let $P_{\mathbf f}$ be the corresponding set of the partition. 

We claim that $P_{\mathbf f}$ is the set of lattice points in 
a(n inverted) simplex of sidelength $k$ and dimension $\iota(\mathbf f)-1=|E|-1-\bar\iota(\mathbf f)$. More precisely, if we let
\[\nabla_{\mathbf f}=\conv\{\,-\mathbf i_{\{e\}}\mid e\in E\text{ is internally active with respect to }\mathbf f\,\},\]
then we have
\begin{equation}\label{eq:partition}
P_{\mathbf f}=\mathbf f+(k\nabla_{\mathbf f}\cap\Z^E).
\end{equation}
As this is the key idea of the proof we included Figure \ref{fig:minkowski} to illustrate it. The solid black dots of the figure represent the hypertrees of some hypergraph $\mathscr H=(V,E)$ of three hyperedges $e_0$, $e_1$, $e_2$. They are included in the regular triangle with vertices $(|V|-1)\mathbf i_{\{e_i\}}$, $i=0,1,2$, although we changed those labels to $e_i$ for simplicity (for the example we let $|V|=17$ but that is not important). Instead of $\conv B_{\mathscr H}+k\nabla_E$, we constructed its integer translate $\conv B_{\mathscr H}+k(\nabla_E+\mathbf i_{\{e_0\}})$ (using $k=5$) which is in the same (hyper)plane as $B_{\mathscr H}$. Figure \ref{fig:minkowski} also shows the partition that corresponds to the order $e_0<e_1<e_2$. It consists of the lattice points in
\begin{itemize}
\item a single triangle shaded grey, one of whose vertices is the unique hypertree with $\bar\iota=0$;
\item one line segment for each hypertree with $\bar\iota=1$, which are the hypertrees represented by the medium-sized black dots;
\item a singleton for each hypertree with $\bar\iota=2$.
\end{itemize}

\begin{figure}[htbp] 
\labellist
\small
\pinlabel $e_0$ at 60 150
\pinlabel $e_1$ at 680 150
\pinlabel $e_2$ at 370 680
\endlabellist
   \centering
   \includegraphics[width=2.5in]{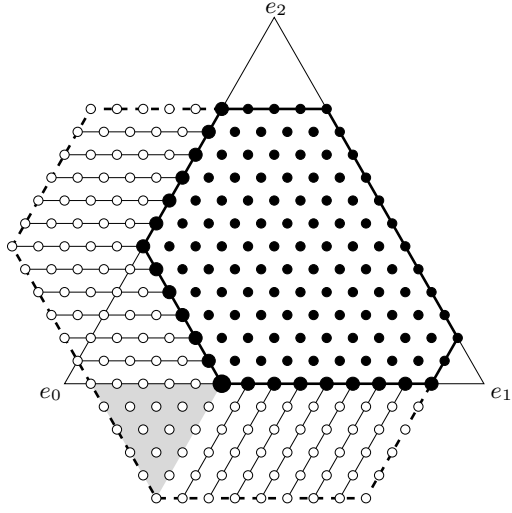} 
   \caption{Minkowski sum of a hypertree polytope and an inverted simplex.}
   \label{fig:minkowski}
\end{figure}

To prove \eqref{eq:partition}, we establish the two-way inclusion as follows.

Proof that $\mathbf f+(k\nabla_{\mathbf f}\cap\Z^E)\subset P_{\mathbf f}$. Assume that $\mathbf u\in\mathbf f+(k\nabla_{\mathbf f}\cap\Z^E)$ is also an element of $\mathbf g+(k\nabla_E\cap\Z^E)$ for some hypertree $\mathbf g$. We claim that $\mathbf f\le\mathbf g$ lexicographically. Suppose that $\mathbf f\ne\mathbf g$ and let $e$ denote the smallest hyperedge where they differ. Let $A$ be the set of hyperedges that are larger than $e$ and internally active with respect to $\mathbf f$. Since $\mathbf f$ is such that no element of $A$ can transfer valence to $e$, for all $a\in A$ there exists a set $T_a$ of hyperedges that is tight at $\mathbf f$, contains $e$, and does not contain $a$. 
Put $T=\cap_{a\in A}T_a$ if $A\ne\varnothing$ and $T=E$ otherwise. Then the set $T$ is tight at $\mathbf f$ (by Lemma \ref{lem:tight}) and contains $e$. Hence if we had $\mathbf g(e)>\mathbf f(e)$ then $T$ would need to have an element $e'$ so that $\mathbf g(e')<\mathbf f(e')$. Because of the way we chose $e$, this $e'$ would obviously satisfy $e<e'$ and such elements of $T$ are internally inactive with respect to $\mathbf f$. Now from $\mathbf u\in\mathbf f+(k\nabla_{\mathbf f}\cap\Z^E)$ it follows that $\mathbf u(e')=\mathbf f(e')$; on the other hand $\mathbf u\in\mathbf g+(k\nabla_E\cap\Z^E)$ implies $\mathbf g(e')\ge\mathbf u(e')=\mathbf f(e')$, which is a contradiction. Therefore $\mathbf g(e)<\mathbf f(e)$ and thus $\mathbf f<\mathbf g$ as claimed.

Proof that $P_{\mathbf f}\subset\mathbf f+(k\nabla_{\mathbf f}\cap\Z^E)$. Assume the contrary, namely that a vector $\mathbf u\in P_{\mathbf f}$ exists so that $\mathbf u(e)<\mathbf f(e)$ for a hyperedge $e$ that is internally inactive with respect to $\mathbf f$. Let $\mathbf g$ be a hypertree that results from $\mathbf f$ by a single transfer of valence from $e$ to a smaller hyperedge. Then $\mathbf g<\mathbf f$ lexicographically and $\mathbf u\in\mathbf g+(k\nabla_E\cap\Z^E)$ since each component of $\mathbf u$ is bounded above by the corresponding component of $\mathbf g$. This contradicts the definition of $P_{\mathbf f}$ and thus concludes the proof of \eqref{eq:partition}. 

From \eqref{eq:partition} it follows that if the interior polynomial, using the given order, of $\mathscr H$ is $I_{\mathscr H}(\xi)=a_0+a_1\xi+a_2\xi^2+\cdots$, then for all $k$, the polytope $(\conv B_{\mathscr H})+k\nabla_E$ contains
\[a_0{k+|E|-1 \choose |E|-1}+a_1{k+|E|-2 \choose |E|-2}+a_2{k+|E|-3 \choose |E|-3}+\cdots\]
lattice points. Since the binomial coefficients $\{\,{k+n \choose n}\mid n=0,1,2,\ldots\,\}$ constitute a basis of the polynomial ring $\Q[k]$ (this is obvious from the fact that the degree of ${k+n \choose n}$ is $n$) and $(\conv B_{\mathscr H})+k\nabla_E$ does not depend on the order, the order-independence of $a_0,a_1,a_2,\ldots$ follows.

There is a similar proof for the order-independence of the exterior polynomial \cite{hypertutte} using the standard simplex $\Delta_E=-\nabla_E$ instead of $\nabla_E$. Both arguments extend to the case of integer polymatroids without difficulty to yield the following result.

\begin{thm}\label{thm:mink}
Let $S$ be a finite ground set and $\mu\colon\mathscr P(S)\to\Z$ a submodular and non-decreasing function on its power set. Let 
\[P_\mu=\{\,\mathbf x\in\R^S\mid \mathbf x\ge\mathbf0\text{ and }\mathbf x\cdot\mathbf i_{U}\le\mu(U)\text{ for all }U\subset S\,\}\]
be the polymatroid associated to $\mu$ and $B_\mu=\{\,\mathbf x\in P_\mu\mid\mathbf x\cdot\mathbf i_S=\mu(S)\,\}$ its base polytope. Then the interior and exterior polynomials $I_\mu$ and $X_\mu$ \cite{hypertutte}, in the standard basis $1,\id,\id^2,\ldots,\id^{|S|-1}$, have the same coefficients as the polynomials (in $k$)
\[\left|(B_\mu+k\nabla_S)\cap\Z^S\right| \quad \text{and} \quad \left|(B_\mu+k\Delta_S)\cap\Z^S\right|,\]
respectively, in the basis ${k+|S|-1 \choose |S|-1},\ldots,{k+2 \choose 2},k+1,1$.
\end{thm}


Theorem \ref{thm:mink} offers new interpretations of the specializations $T(x,1)$ and $T(1,y)$ of the Tutte polynomial of a matroid (graph) in terms of lattice point counts in Minkowski sums of the base polytope (spanning tree polytope) and a simplex. Interestingly though, this does not lead to a new interpretation of the two-variable polynomial $T(x,y)$. Tempting as it may be to count lattice points in the base polytope plus $k$ times the inverted simplex plus $l$ times the standard simplex, the resulting two-variable polynomial (when evaluated for a matroid/graph) will be 
fundamentally different from (less subtle than) $T(x,y)$.


\section{The root polytope}\label{sec:root}

In this section we recall the general notions of $f$-vector and $h$-vector and discuss the root polytope of a bipartite graph. Most of the results 
have also appeared in \cite{alex} but some, notably Proposition \ref{pro:facet} and Theorem \ref{thm:h=E}, are new.

\subsection{Triangulations}

A \emph{triangulation} of a polytope $Q$ is a collection of maximal simplices, each spanned by the vertices of $Q$, so that each two intersect in a common face and their union is $Q$. A triangulation is an instance of a pure simplicial complex, i.e., one in which all maximal simplices have the same dimension. To a $d$-dimensional simplicial complex, it is customary to associate the \emph{$f$-vector}
\begin{equation}\label{eq:f}
f(y)=y^{d+1}+f_0\,y^{d}+f_1\,y^{d-1}+\cdots+f_{d-2}\,y^2+f_{d-1}\,y+f_d,
\end{equation}
where $f_k$, for $k\ge0$, is the number of $k$-dimensional simplices in the complex. The \emph{$h$-vector} of the same complex is defined as $h(x)=f(x-1)$. The latter notion becomes significant (for example, it has positive coefficients)  for so-called \emph{shellable complexes}, that is complexes with a shelling order. Here a \emph{shelling order} of a pure simplicial complex, $\sigma_1<\sigma_2<\cdots<\sigma_{f_d}$, lists the maximal simplices in such a way that each $\sigma_i$, $i\ge1$, intersects the set~$\sigma_1\cup\cdots\cup\sigma_{i-1}$ in a union of $c_i$ codimension one faces. We always have $c_1=0$ but assume as part of the definition that $c_i\ge1$ for $i\ge2$. Whether such an order exists is a subtle question, but when it does, it is not hard to show \cite{swartz} that 
\begin{equation}\label{eq:ftoh}
h(x)=f(x-1)=\sum_{i=1}^{f_d}x^{d+1-c_i}.
\end{equation}

\begin{Def}
Let $G$ be a 
bipartite graph with color classes $E$ and $V$. In the space $\R^E\oplus\R^V$ we consider all vectors of the form $\mathbf i_{\{e\}}+\mathbf i_{\{v\}}$, where $e\in E$, $v\in V$, and $ev$ is an edge in $G$. We denote their convex hull by $Q_G$ and call 
it the \emph{root polytope} of $G$. 
\end{Def}

The sum of the coordinates for each vertex of $Q_G$ is $2$. Furthermore, the sum of the $E$-coordinates equals that of the $V$-coordinates. When $G$ is connected there are no more affine relations and the dimension of $Q_G$ is $|E|+|V|-2$.

In \cite{alex} the vertices of $Q_G$ are given in the form $\mathbf i_{\{e\}}-\mathbf i_{\{v\}}$. The two versions are isometric via multiplying each $V$-coordinate by $-1$.

\begin{ex}\label{ex:k23}
The root polytope of the complete bipartite graph with color classes $E$ and $V$ is, by definition, the direct product of the unit simplices $\Delta_E$ and $\Delta_V$. In Figure \ref{fig:k23} we picture it in the case when $|E|=3$ and $|V|=2$, mainly to set notation for Example \ref{ex:triang}. Note that we replaced the symbols $\mathbf i_{\{x\}}$ with $\mathbf x$ for better readability.
\begin{figure}[h!] 
\labellist
\tiny
\pinlabel $v_0$ at 280 70
\pinlabel $v_1$ at 280 582
\pinlabel $e_0$ at 25 326
\pinlabel $e_1$ at 280 326
\pinlabel $e_2$ at 535 326
\pinlabel ${\mathbf e}_0+{\mathbf v}_0$ at 810 170
\pinlabel ${\mathbf e}_0+{\mathbf v}_1$ at 810 490
\pinlabel ${\mathbf e}_1+{\mathbf v}_0$ at 1470 0
\pinlabel ${\mathbf e}_1+{\mathbf v}_1$ at 1470 320
\pinlabel ${\mathbf e}_2+{\mathbf v}_0$ at 1150 330
\pinlabel ${\mathbf e}_2+{\mathbf v}_1$ at 1150 650
\endlabellist
   \includegraphics[width=3in]{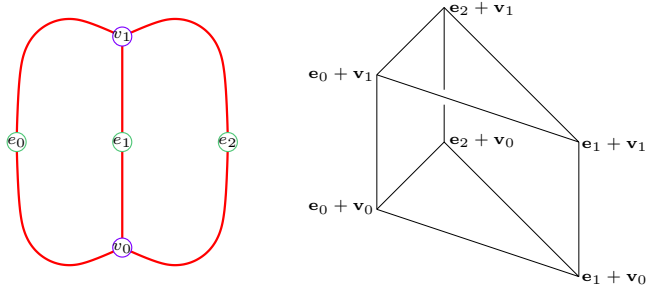} 
   \caption{The complete bipartite graph $K_{2,3}$ and its root polytope.}
   \label{fig:k23}
\end{figure}
\end{ex}

The vertices of $Q_G$ obviously correspond to edges of $G$; it is also not hard to verify \cite[Lemma~12.5]{alex} that a set of vertices is affinely independent if and only if the corresponding set of edges is cycle-free. In other words, simplices in $Q_G$ (spanned by vertices) correspond to forests in $G$. In particular, maximal simplices in $Q_G$ are in a one-to-one correspondence with spanning trees in $G$. Next we establish an elementary property of these simplices.


\begin{lemma}\label{lem:unimod}
If the vertices of the simplex $\sigma$ are the vertices $\mathbf v_1,\ldots,\mathbf v_m$ of $Q_G$, then for each positive integer $s$, the set of integer points in $s\cdot\sigma$ agrees with the set
\[A=\{\,\lambda_1\mathbf v_1+\cdots+\lambda_m\mathbf v_m \mid \lambda_1,\ldots,\lambda_m\in\N, \lambda_1+\cdots+\lambda_m=s\,\}.\]
\end{lemma}

\begin{proof}
Each point of $s\cdot\sigma$ can be uniquely written as $\sum_{i=1}^m\lambda_i\mathbf v_i$, where the $\lambda_i$ are non-negative reals summing to $s$. Since the $\mathbf v_i$ are integer vectors, the relation $A\subset(s\cdot\sigma)\cap(\Z^E\oplus\Z^V)$ is clear. Conversely, let $\mathbf p$ be an integer point in $s\cdot\sigma$. We need to check that the unique solution (in $\lambda_1,\ldots,\lambda_m$) of $\sum_{i=1}^m\lambda_i\mathbf v_i=\mathbf p$ is integer.

In terms of the forest $\Sigma$ that corresponds to $\sigma$, the unknowns $\lambda_i$ are associated to the edges of $\Sigma$ and the coordinates of $\mathbf p$ are associated to the vertices. The condition is that the value at each vertex be the sum of the values on the adjacent edges. Now this system of equations is in `triangular form': $\Sigma$ has at least one degree $1$ vertex and the corresponding equation forces the unknown on the adjacent edge to take an integer value. Then we can remove this vertex-edge pair and look for another degree $1$ vertex. This way a simple inductive argument can be constructed which shows that if the system of equations has a solution (which we assumed) then that solution needs to be integer.
\end{proof}

As to the relative position of two maximal simplices in $Q_G$, we recall the following basic observation.

\begin{lemma}[{\cite[Lemma 12.6]{alex}}]\label{lem:compa}
Let $\Gamma_1$ and $\Gamma_2$ be spanning trees in $G$. The following two statements are equivalent.
\begin{enumerate}
\item\label{odin} The simplices in $Q_G$ that correspond to the $\Gamma_i$ intersect in a common face.
\item\label{dva} There does not exist a cycle $\varepsilon_1,\varepsilon_2,\ldots,\varepsilon_{2k}$ of edges in $G$, where $k\ge2$, so that all odd-index edges are from $\Gamma_1$ and all even-index edges are from $\Gamma_2$.
\end{enumerate}
\end{lemma}

If two maximal simplices (spanning trees) satisfy the equivalent conditions of Lemma~\ref{lem:compa} then we call them \emph{compatible}. A triangulation of $Q_G$ is then a collection of pairwise compatible maximal simplices whose union is $Q_G$. Since all maximal simplices have the same volume \cite[Lemma 12.5]{alex}, each triangulation of $Q_G$ consists of the same number of them. Theorem \ref{thm:h=E} below shows that the triangulations have much more in common.

\subsection{Facets}

Our next goal is to describe the root polytope by a system of linear inequalities. In other words, we shall discuss the facets (codimension one strata of the boundary) of $Q_G$. These turn out to be in a one-to-one correspondence with certain cuts in $G$, as follows. 
A \emph{cut} in a graph is a set of edges that is obtained by splitting the set of vertices into the disjoint union of two subsets, and then taking the set of edges between the two subsets. A non-empty cut is \emph{minimal} if it does not contain any other non-empty cuts. For example, all \emph{star-cuts} (when the vertex set is split into a non-isolated singleton and the rest) are minimal.

In a bipartite graph $G$ with color classes $E$ and $V$ we may speak of \emph{directed cuts}, that is cuts which arise from a splitting $V\cup E=S\cup T$ so that each edge in the cut is adjacent to $S$ at an element of $V$ (and then to $T$ at an element of $E$). 
In other words, $G$ has no edges between $E\cap S$ and $V\cap T$. 
For instance, star-cuts in a bipartite graph are directed.

\begin{lemma}\label{lem:support}
Let the splitting $V\cup E=S\cup T$ induce the non-empty directed cut $C$ in the bipartite graph $G$. Then the root polytope $Q_G$ has a supporting hyperplane which does not contain any of the vertices corresponding to elements of $C$ but contains all other vertices.
\end{lemma}

\begin{proof}
Assume that all the edges of $C$ are adjacent to $S$ at elements of $V$. The linear functional $\lambda_C\colon\R^E\oplus\R^V\to\R$ which is defined to take the value $1$ at $\mathbf i_{\{x\}}$ if $x\in(S\cap V)\cup(T\cap E)$ and the value $-1$ if $x\in(S\cap E)\cup(T\cap V)$ is such that its value at vertices of $Q_G$ corresponding to elements of $C$ is $2$, while at all other vertices $\lambda_C$ vanishes. Thus 
the hyperplane $\Pi_C=\ker\lambda_C$ 
has the required properties.
\end{proof}


%
%

\begin{prop}\label{pro:facet}
Let $G$ be a connected bipartite graph. 
For each minimal directed cut $C$ of $G$, the supporting hyperplane $\Pi_C$ of Lemma \ref{lem:support} intersects the root polytope~$Q_G$ in a facet. Furthermore, each facet of $Q_G$ arises this way.
\end{prop}

\begin{proof}
Let the minimal directed cut $C$ belong to the splitting $V\cup E=S\cup T$. The subgraphs of $G$ induced by $S$ and $T$, respectively, are connected for otherwise $C$ would not be minimal. By picking spanning trees in each, we obtain a two-component forest in $G$ whose edge set is disjoint from $C$. The vertices of $Q_G$ that correspond to edges of the forest are affinely independent, i.e., they form an $(|E|+|V|-3)$-dimensional simplex. The fact that this simplex lies in $\Pi_C$ proves the first claim.

Conversely, every facet of $Q_G$ contains some $|E|+|V|-2$ affinely independent vertices which span a codimension $1$ (rel.\ $Q_G$) simplex $\sigma$. The cycle-free subgraph of $G$ to which these correspond is a two-component forest~$F$. The vertex sets of the two components define a minimal cut $C$ in $G$. The cut $C$ is directed because if it was not then $F$ would have two extensions to spanning trees of $G$ that satisfy the compatibility condition of Lemma \ref{lem:compa} --- but that would mean that the corresponding maximal simplices lie on opposite sides of $\sigma$, which is clearly impossible. Finally, as $\Pi_C$ contains $\sigma$, it has to intersect $Q_G$ in the given facet.
\end{proof}

\begin{ex}
The root polytope of Figure \ref{fig:k23} has five facets which correspond (in the manner described in Lemma \ref{lem:support}) to the star-cuts at each of the five vertices. In this case the graph has no other minimal directed cuts.
\end{ex}

\subsection{Ehrhart polynomials}

For $s\in\N$ let us define $\varepsilon_G(s)=|(s\cdot Q_G)\cap(\Z^E\oplus\Z^V)|$. This is well known to be a polynomial 
in $s$ (which can then be extended to all $s\in\C$), called the \emph{Ehrhart polynomial} of $Q_G$. Let us write $\varepsilon_G$ as in \eqref{eq:ehrhart}, 
\begin{equation}\label{eq:again}
\varepsilon_G(s)=\sum_{k=0}^{|E|+|V|-2}a_kC_k(s),
\end{equation}
using the binomial coefficients $C_k(s)={s+|E|+|V|-2-k \choose |E|+|V|-2}$ and rational numbers $a_k$. There is a unique such expression because the degree of $\varepsilon_G$ is the dimension of $Q_G$ and the following observation.

\begin{lemma}\label{lem:basis}
The expressions $C_0(s),\ldots,C_{|E|+|V|-2}(s)$ form a basis over $\Q$ in the space $P_{|E|+|V|-2}$ of polynomials with rational coefficients and of degree no greater than $|E|+|V|-2$.
\end{lemma}

Why we prefer this basis will be made clear by Theorem \ref{thm:h=E}. Then in Remark~\ref{rem:shell} we explain how one naturally arrives at it.

\begin{proof}
The degree of each $C_k(s)$ is exactly $|E|+|V|-2$. As their number equals the dimension, it suffices to show that the $C_k(s)$ span $P_{|E|+|V|-2}$. Since an element of $P_{|E|+|V|-2}$ is determined by its values at $0,1,2,\ldots,|E|+|V|-2$, it suffices to arbitrarily fix rational numbers $r_0,r_1,\ldots,r_{|E|+|V|-2}$ and to find a linear combination of the $C_k(s)$ that takes the value $r_j$ at $j$ for each $j=0,1,\ldots,|E|+|V|-2$.

The roots of $C_k(s)$ are the consecutive integers $k-|E|-|V|+2,\ldots,k-1$. In particular, $C_0(s)$ is the only polynomial among the $C_k(s)$ that does not vanish at $0$. This determines the coefficient of $C_0(s)$ in the desired linear combination. 
Since only $C_0(s)$ and $C_1(s)$ take non-zero values at $1$, the coefficient of $C_1(s)$ also gets 
determined. By a trivial induction proof, the same is true for all coefficients and the resulting linear combination obviously satisfies our requirement.
\end{proof}

\begin{rmk}\label{rem:ehr}
The set $A$ of Lemma \ref{lem:unimod} has cardinality ${s+m-1 \choose m-1}$. Thus the simplices $\sigma$ that appear in Lemma \ref{lem:unimod} have Ehrhart polynomials $\varepsilon_\sigma(s)={s+\dim\sigma \choose \dim\sigma}$. It follows that the number of lattice points in $s$ times the \emph{relative interior} of $\sigma$ is ${s-1 \choose \dim\sigma}$. This is the same observation that the standard proof of 
Ehrhart reciprocity is based on and in any case it follows from that principle since ${s-1 \choose \dim\sigma}=(-1)^{\dim\sigma}{-s+\dim\sigma \choose \dim\sigma}$.
\end{rmk}

\begin{thm}\label{thm:h=E}
If $\varepsilon_G$ is the Ehrhart polynomial of $Q_G$ (as described in \eqref{eq:again}), then the $h$-vector $h$ of any triangulation of $Q_G$ satisfies
\begin{equation}\label{eq:h=E}
x^{|E|+|V|-1}h(x^{-1})=a_0+a_1\,x+a_2\,x^2+\cdots+a_{|E|+|V|-2}\,x^{|E|+|V|-2}.
\end{equation}
In particular, all triangulations of $Q_G$ share the same $h$-vector.
\end{thm}

We will denote the common $h$-vector described in Theorem \ref{thm:h=E} by $h_G$. 

It is obvious from \eqref{eq:f} that the left hand side of \eqref{eq:h=E} is a polynomial. Since the Euler characteristic of a convex polytope 
is $1$, we have 
\[h(0)=f(-1)=(-1)^{d+1}+\sum_{l=0}^d(-1)^{d-l}f_l=(-1)^{d+1}\left(1-\sum_{l=0}^d(-1)^lf_l\right)=0,\] 
implying that both sides of \eqref{eq:h=E} are in $P_{|E|+|V|-2}$. Their degree is in fact much lower than $|E|+|V|-2$, but we will only see that after proving Theorem~\ref{thm:main}: indeed by \cite[Proposition 6.1]{hypertutte} the degree of the interior polynomial is at most $\min\{\,|E|,|V|\,\}-1$. This does not contradict the requirement that the degree of the Ehrhart polynomial $\varepsilon_G$ be $\dim Q_G=|E|+|V|-2$ because each $C_k(s)$ in \eqref{eq:again} has that degree.


\begin{proof}
Let us put $|E|+|V|-2=d$. Fix a triangulation $\mathscr T$ of $Q_G$ with $h$-vector $h$ and corresponding $f$-vector $f(y)=h(y+1)$, cf.\ \eqref{eq:f}. That means that 
\begin{multline*}
h(x)=f(x-1)=(x-1)^{d+1}+\sum_{l=0}^{d}f_l(x-1)^{d-l}\\
=\sum_{i=1}^{d+1}\left((-1)^{d+1-i}{d+1 \choose i}+\sum_{l=0}^{d-i}f_l(-1)^{d-l-i}{d-l \choose i}\right)x^i,
\end{multline*}
where we used the fact that $h(0)=0$. Thus the left hand side of \eqref{eq:h=E} (putting $i=d+1-k$ above) is
\begin{equation}\label{eq:LHS}
\sum_{k=0}^d\left((-1)^k{d+1 \choose d+1-k}+\sum_{l=0}^{k-1}f_l(-1)^{k-l-1}{d-l \choose d+1-k}\right)x^k.
\end{equation}

In order to address the right hand side of \eqref{eq:h=E}, first we 
express ${s-1 \choose l}$ (for $l=0,1,\ldots,d$) in the basis $C_k(s)={s+d-k \choose d}$ ($k=0,1,\ldots,d$) of Lemma~\ref{lem:basis}. (Why we need this will be made clear by \eqref{eq:ehr-part} below.) To do so we will rely on the identity ${n \choose m}=\sum_{i=0}^p(-1)^i{p \choose i}{n+p-i \choose m+p}$. 
(This holds for any integers $n,m$ and non-negative integer $p$. It is very easy to prove by induction on $p$ based on just the basic relation in Pascal's triangle.) Applying it directly to ${s-1 \choose l}$ with $p=d-l$ gives
\begin{align*}
{s-1 \choose l}={}&\sum_{i=0}^{d-l}(-1)^i{d-l \choose i}{s-1+d-l-i \choose l+d-l}\\
={}&\sum_{i=0}^{d-l-1}(-1)^i{d-l \choose i}{s-1+d-l-i \choose d}+(-1)^{d-l}{s-1 \choose d}.
\end{align*}
Here the last term is not one of our basis elements. To express it in the basis we apply our identity again, this time to ${s-1 \choose -1}=0$ with $p=d+1$:
\begin{align*}
{s-1 \choose l}={}&\sum_{i=0}^{d-l-1}(-1)^i{d-l \choose i}{s+d-(l+i+1) \choose d}\\
&+(-1)^{d-l}(-1)^{d+1}\left(-\sum_{j=0}^d(-1)^j{d+1 \choose j}{s+d-j \choose d}\right)\\
={}&\sum_{k=0}^d\left((-1)^{k-l-1}{d-l \choose k-l-1}+(-1)^{l+k}{d+1 \choose k}\right){s+d-k \choose d}.
\end{align*}
If $k\le l$ then of course ${d-l \choose k-l-1}=0$.

Now as the relative interiors of the simplices of $\mathscr T$ partition $Q_G$, by Remark \ref{rem:ehr} we obtain
\begin{multline}\label{eq:ehr-part}
\varepsilon_G(s)=\sum_{\text{simplices }\sigma\text{ in }\mathscr T}{s-1 \choose \dim\sigma}=\sum_{l=0}^{d}f_l{s-1 \choose l}\\
=\sum_{k=0}^d\sum_{l=0}^d\left((-1)^{k-l-1}{d-l \choose k-l-1}+(-1)^{l+k}{d+1 \choose k}\right)\cdot f_l\cdot{s+d-k \choose d},
\end{multline}
that is that in \eqref{eq:again} (and the right hand side of \eqref{eq:h=E}) we have 
\begin{align}
\nonumber
a_k={}&\sum_{l=0}^d\left((-1)^{k-l-1}{d-l \choose k-l-1}+(-1)^{l+k}{d+1 \choose k}\right)\cdot f_l\\
\nonumber
={}&(-1)^k{d+1 \choose k}\sum_{l=0}^d(-1)^lf_l+\sum_{l=0}^{k-1}(-1)^{k-l-1}{d-l \choose k-l-1}f_l\\
\label{eq:coefficients}
={}&(-1)^k{d+1 \choose k}+\sum_{l=0}^{k-1}(-1)^{k-l-1}{d-l \choose k-l-1}f_l.
\end{align}
Comparing the last formula to \eqref{eq:LHS} completes the proof.
\end{proof}

\begin{rmk}\label{rem:coefficients}
From \eqref{eq:coefficients} we see that $a_0=1$ and $a_1=-(d+1)+f_0$. Here $f_0$ is the number of edges in $G$ and $d+1=|E|+|V|-1$ is the number of edges in a spanning tree of $G$, i.e., we have $a_1=b_1(G)$. These values are consistent with \cite[Proposition~6.2 and Theorem 6.3]{hypertutte} and our expectation (cf.\ Theorem \ref{thm:main}) that they be the two lowest-degree coefficients in the common interior polynomial of the hypergraphs $(V,E)$ and $(E,V)$. As to $a_2$, see Proposition \ref{pro:a_2}.
\end{rmk}

\begin{rmk}\label{rem:shell}
The proof of Theorem \ref{thm:h=E} becomes far more transparent if we assume $\mathscr T$ to be shellable\footnote{It seems unlikely for all triangulations of $Q_G$ to be shellable, but at present we do not know any counterexamples.}. This also makes the appearance of $C_k(s)$ more natural. Indeed if there is a shelling order with the quantities $c_i$ as in \eqref{eq:ftoh}, then the left hand side of \eqref{eq:h=E} is $\sum x^{c_i}$ where the summation is over all maximal simplices. We may also count lattice points in $s\cdot Q_G$ using the shelling order, going over the maximal simplices one-by-one. If we do so then the $i$'th simplex contributes lattice points as described in Lemma \ref{lem:unimod}, except that the points along $c_i$ of its facets have already been counted. Geometrically that means that what is left to count are the lattice points in a simplex of full dimension $|E|+|V|-2$ but with reduced sidelength $s-c_i$. (By sidelength we mean the number of lattice points along a side of the simplex minus one. In the absence of shellability one is forced to work with lower-dimensional simplices too, and that causes complications.) The number of those points is exactly $C_{c_i}(s)$. Therefore $a_k$, that is the number of times $C_k(s)$ appears in $\varepsilon_G(s)$, is equal to the number of maximal simplices with $c_i=k$, and that in turn is exactly the coefficient of $x^k$ on the left hand side.
\end{rmk}

\subsection{Cross-sections}\label{ssec:cross}

A remarkable property of triangulations of the root polytope is that they `pair up' hypertrees in the hypergraph $\mathscr H=(V,E)$ and its abstract dual (transpose) $\overline{\mathscr H}=(E,V)$. To set up the correspondence, we recall the following facts from \cite{alex}. First, $Q_G$ naturally projects onto the standard unit simplices $\Delta_E\subset\R^E$ and $\Delta_V\subset\R^V$. The preimage of the barycenter $\mathbf i_V/|V|\in\Delta_V$ under the second projection $\mathrm{pr}_2\colon Q_G\to\Delta_V$ is the $(|E|-1)$-dimensional Minkowski sum 
\begin{equation}\label{eq:cross}
S_E=\frac1{|V|}\left(\sum_{v\in V}\Delta_v\right)+\frac1{|V|}\cdot\mathbf i_V,
\end{equation}
where $\Delta_v\subset\Delta_E\subset\R^E$ is the convex hull of the set $\{\,\mathbf i_{\{e\}}\mid ev\text{ is an edge in }G\,\}$. In particular, for a maximal simplex $\gamma$ in $Q_G$ that corresponds to the spanning tree $\Gamma$ in $G$, we have
\begin{equation}\label{eq:smallcell}
\gamma\cap S_E=\frac1{|V|}\left(\sum_{v\in V}\Delta^\Gamma_v\right)+\frac1{|V|}\cdot\mathbf i_V,
\end{equation}
where $\Delta^\Gamma_v$ is the face of $\Delta_v$ spanned by unit vectors corresponding to neighbors of $v$ in $\Gamma$. Moreover, a quick dimension count shows that the Minkowski sum in the latter formula is a direct sum, i.e., its elements have unique representations as sums of one vector from each summand. Below we will sometimes replace $S_E$ with its homothetic image\footnote{Both $S_E$ and $P_E$ belong, as rather special cases, to the class of \emph{generalized permutohedra} \cite{alex}.} $P_E=\sum_{v\in V}\Delta_v$ so that our formulas look simpler, however we will continue to think of $P_E$ as a cross-section of $Q_G$. We will refer to the set described in \eqref{eq:smallcell}, as well as to the homothetic set
\begin{equation}\label{eq:bigcell}
M_\Gamma=\sum_{v\in V}\Delta^\Gamma_v\subset P_E,
\end{equation}
as the \emph{Minkowski cell} associated to the spanning tree $\Gamma$.


Next, we recall that $M_\Gamma$ contains a unique translate of the unit simplex $\Delta_E$ and it is $\mathbf f+\Delta_E$ \cite{alex} where $\mathbf f\in\R^E$ is the hypertree in $\mathscr H$ induced by $\Gamma$. (The interior of $M_\Gamma$ is disjoint from other integer translates of $\Delta_E$.) Let us denote the barycenter of this translated simplex with
\begin{equation}\label{eq:mark}
\mathbf f^+=\mathbf f+\frac1{|E|}\cdot\mathbf i_E.
\end{equation} 
We will call both $\mathbf f^+$ and (depending on context) the corresponding point
\begin{equation}\label{eq:smallmarker}
\frac1{|V|}\cdot\mathbf f^++\frac1{|V|}\cdot\mathbf i_V\in\gamma\cap S_E\subset S_E\subset Q_G,
\end{equation}
the \emph{emerald marker} of the simplex $\gamma$. Note that the set of 
emerald markers is a simple dilation of the set $B_{\mathscr H}$ of hypertrees. If markers are meant in the sense of \eqref{eq:mark}, then in fact it is a translation of $B_{\mathscr H}$, so that for instance vectors connecting hypertrees are the same as vectors connecting the corresponding markers. Of course each maximal simplex in $Q_G$ also contains a unique violet marker, i.e., an element of the set $\frac1{|E|}\left(B_{\overline{\mathscr H}}+\frac1{|V|}\mathbf i_V\right)+\frac1{|E|}\cdot\mathbf i_E$.

Now if we fix a triangulation of $Q_G$, then each emerald and violet marker will lie in the interior of a unique simplex; as this simplex has unique markers of each color (which are essentially the hypertrees that are induced by the spanning tree corresponding to the simplex), we see that a bijection is induced this way between $B_{\mathscr H}$ and $B_{\overline{\mathscr H}}$. In Section~\ref{sec:main}, we exploit this picture further to prove our main theorem.

\begin{ex}\label{ex:triang}
We examine the root polytope of the complete bipartite graph $G=K_{2,3}$. Let the emerald color class be $E=\{\,e_0,e_1,e_2\,\}$ along with violet color class $V=\{\,v_0,v_1\,\}$. In this case $Q_G$ is the product of the $2$-simplex $\Delta_E$ and the $1$-simplex $\Delta_V$. In Example \ref{ex:k23} we labeled its vertices (using the symbol $\mathbf x$ in place of $\mathbf i_{\{x\}}$). In the upper left panel of Figure \ref{fig:fallsapart} we indicated the cross-sections $S_E$ and $S_V$, as well as all markers for both colors (three each). We also chose a triangulation of the root polytope and showed the Minkowski cells that occur in the cross-sections. These are just three isometric segments in the one-dimensional cross-section $S_V$, whereas in $S_E$ we get two triangles and a rhombus. By examining the spanning trees corresponding to the three maximal simplices (shown on the right), the interested reader may check the validity of Lemma \ref{lem:compa} and the formula \eqref{eq:bigcell}. 

The $f$- and $h$-vectors of our triangulation are $f(y)=y^4+6y^3+12y^2+10y+3$ and $h(x)=f(x-1)=x^4+2x^3$, respectively. The coefficients of the latter are explained by the fact that the triangulation is shellable so that the second and third maximal simplices are both attached along a single facet, cf.\ \eqref{eq:ftoh}. Comparing the $h$-vector to the interior polynomial of Example \ref{ex:k23poly}, we see that Theorem \ref{thm:main} holds in this case.

\begin{figure}[htbp]
\parbox{.25\linewidth}{
\includegraphics[width=\linewidth]{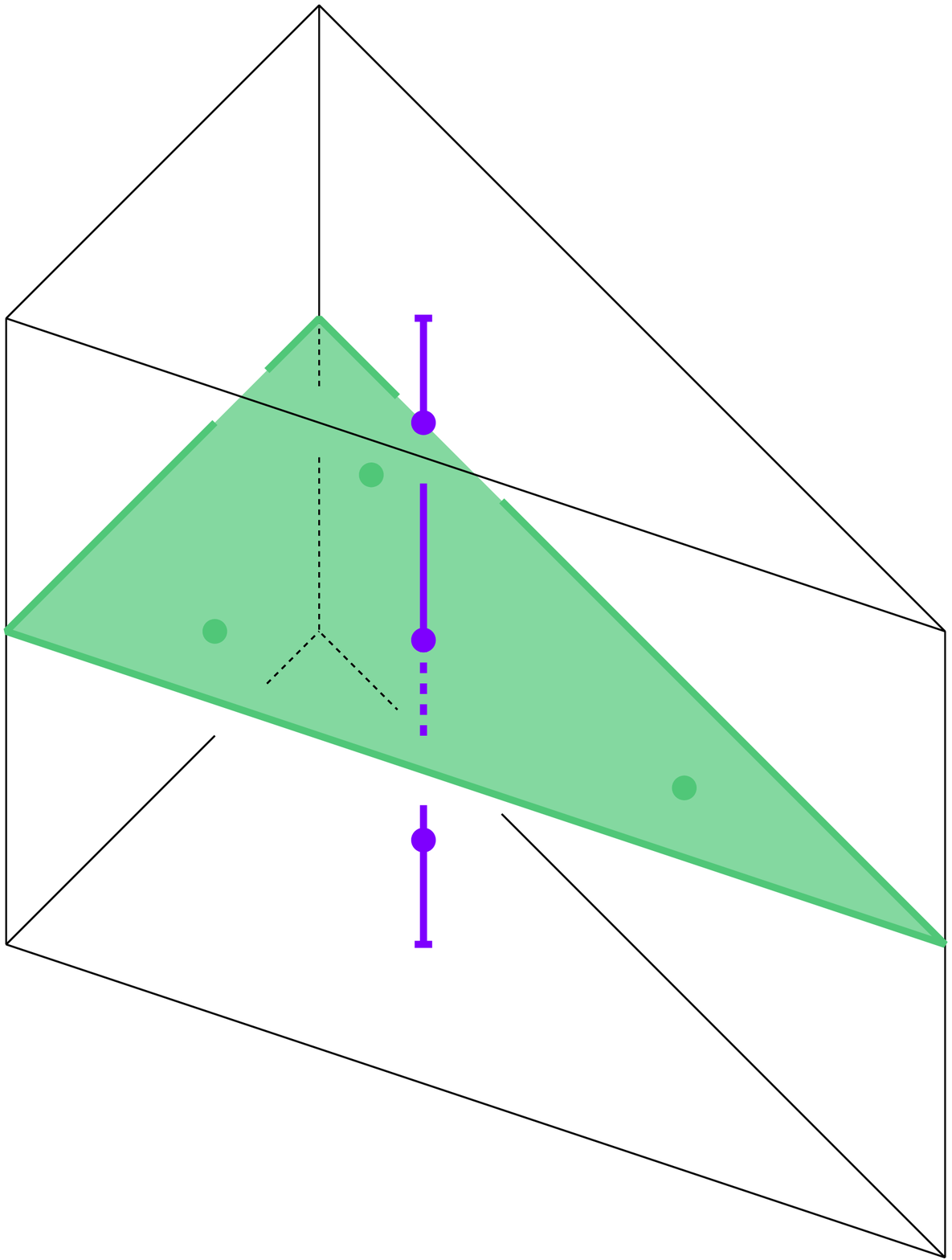}

\includegraphics[width=\linewidth]{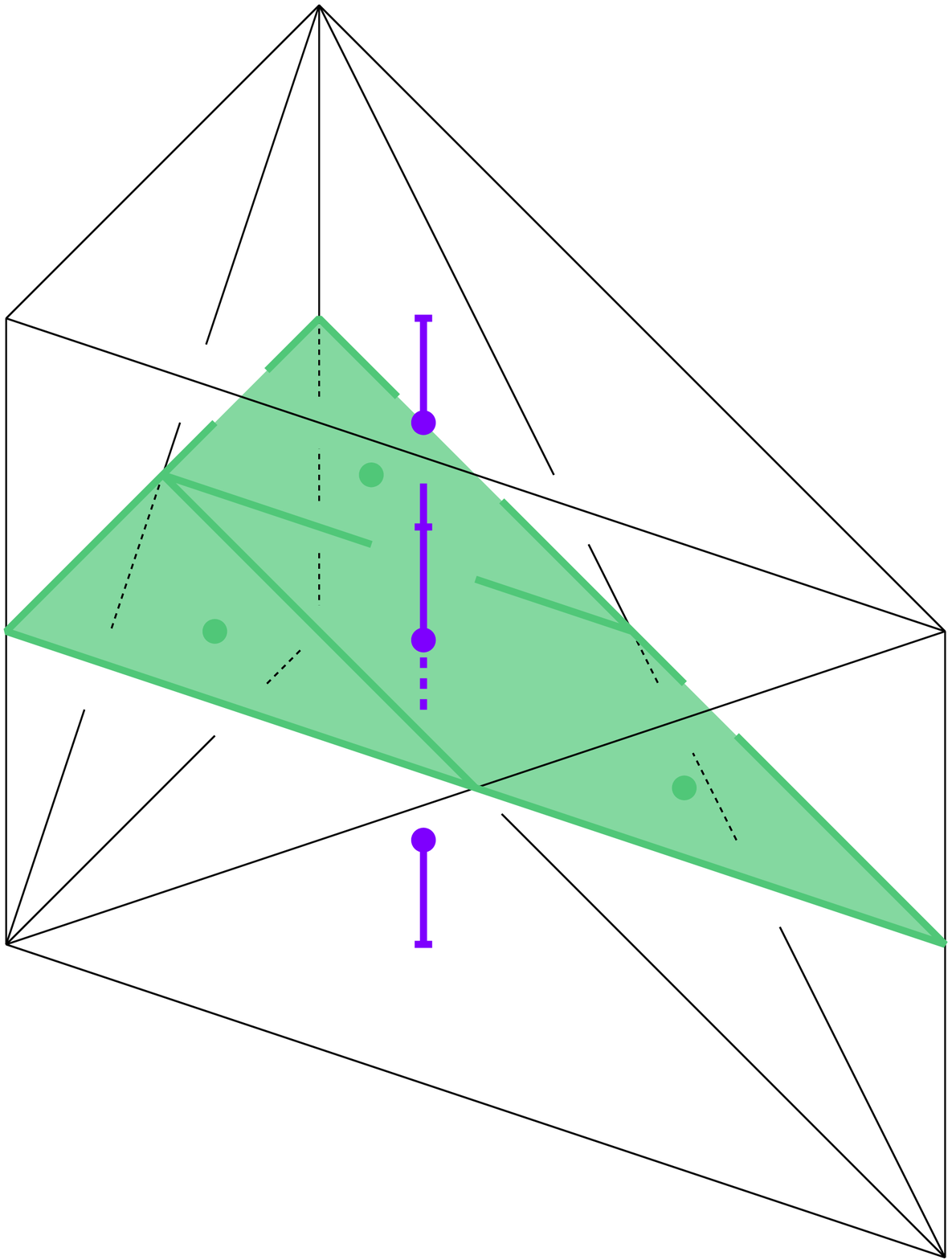}}
\hspace{.1\linewidth}
\raisebox{-1.8in}{\includegraphics[width=.25\linewidth]{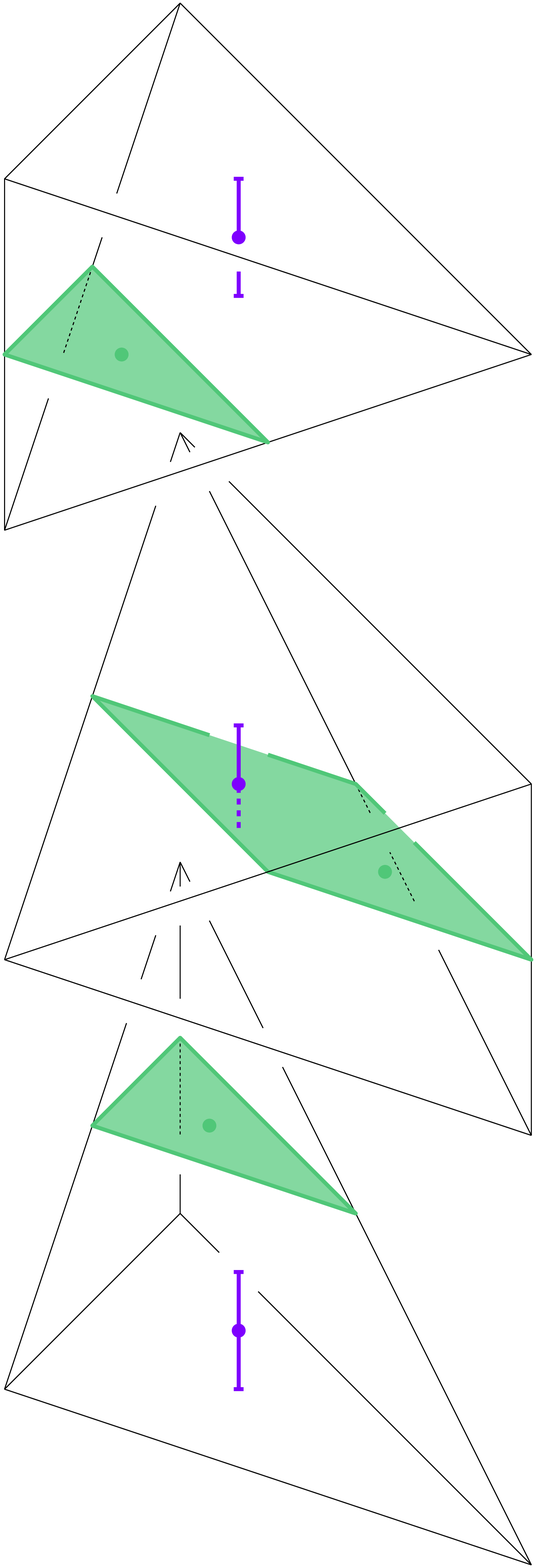}}
\hspace{.1\linewidth}
\raisebox{-1.65in}{\includegraphics[width=.15\linewidth]{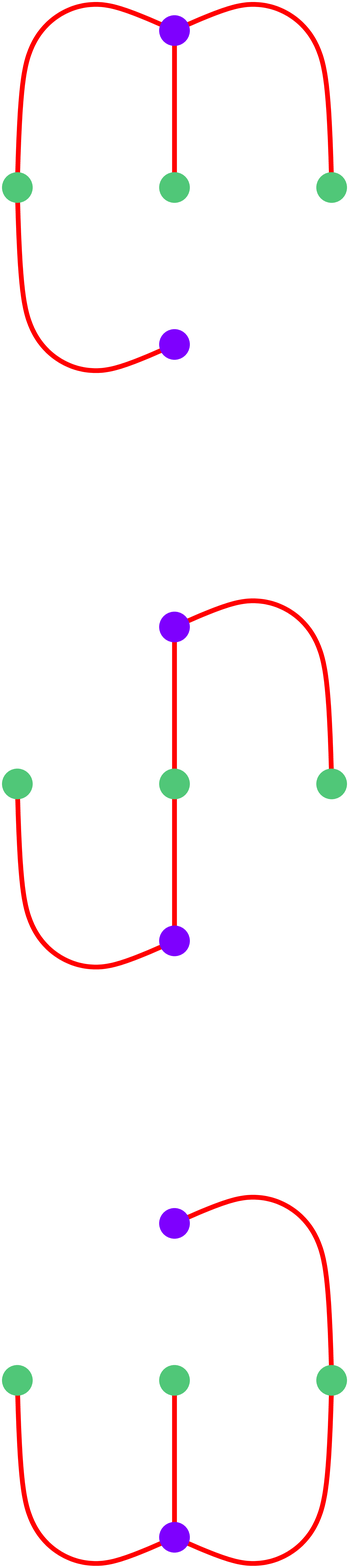}}
\caption{The root polytope of the graph $K_{2,3}$ with cross-sections and markers of both colors (upper left), and a triangulation (lower left and middle) with the corresponding spanning trees (right). Cf.\ 
Figure \ref{fig:k23}.}
\label{fig:fallsapart}
\end{figure}
\end{ex}

We will need one more piece of information about the spanning tree $\Gamma\subset G$, corresponding maximal simplex $\gamma\subset Q_G$, and induced hypertree $\mathbf f\in B_{\mathscr H}$. In \eqref{eq:mark} we gave the coordinates of the emerald marker $\mathbf f^+$ (as it appears in the enlarged cross-section $P_E$) of $\gamma$ in the natural basis of $\R^E$. However it will also be useful for us to describe the marker in terms of the direct sum decomposition \eqref{eq:bigcell} of the Minkowski cell $M_\Gamma$. It is convenient to identify the vertices of $\Delta_v^\Gamma$ with the edges of $\Gamma$ adjacent to $v$. Then the barycentric coordinates of a point in $\Delta_v^\Gamma$ can be thought of as non-negative real numbers (weights) written on these edges, subject to the condition that their sum is $1$. (Since we have $|V|$ such conditions on the $|E|+|V|-1$ weights, this correctly identifies the dimension of $M_\Gamma$ as $|E|-1$.) By taking the sum of the weights assigned to edges adjacent to a vertex $e\in E$, we recover the standard $e$-coordinate in $\R^E$.

\begin{lemma}\label{lem:coordinates}
Let $\Gamma$ be a spanning tree of $G$ that induces the hypertree $\mathbf f$.
The weights (on the edges of $\Gamma$) that produce the emerald marker $\mathbf f^+\in M_\Gamma$ are given as follows. For each $v\in V$, the elements of $E$ are partitioned according to the edge of $G$ adjacent to $v$ through which they can be reached from $v$ by a path in $\Gamma$. Let the weight of the edge be the size of the corresponding set divided by $|E|$. 
\end{lemma}

\begin{proof}
Our assignment obviously satisfies all $|V|$ constraints and furthermore, if the vertex $e\in E$ has degree $d=\mathbf f(e)+1$ in $\Gamma$, then the sum of the weights on the edges adjacent to $e$ is 
\[\frac1{|E|}\left((|E|-1)(d-1)+d\right)=d-1+\frac1{|E|}=\mathbf f(e)+\frac1{|E|}=\mathbf f^+(e),\] 
as claimed. The formula is valid because when we compute the $d$ weights, $e$ appears in all $d$ relevant counts of emerald vertices, whereas all other elements of $E$ appear $d-1$ times.
\end{proof}

\section{Preparatory results}\label{sec:prep}

In this section we prove a few more technical results needed for our main theorem. 
Regarding a polytope, we say that a simplex spanned by some of its vertices is \emph{interior} if it is not part of the boundary of the polytope, i.e., if the (relative) interior points of the simplex are interior points of the polytope. E.g., all maximal simplices are interior but $0$-dimensional simplices (vertices) are never interior.


Let $\Gamma$ be a tree in our usual bipartite graph $G$ of color classes $E$ and $V$, and let $C\subset\Gamma$ be a connected subgraph. If $\varepsilon$ is an edge in $\Gamma\setminus C$, then there is a clear sense in which one endpoint of $\varepsilon$ is closer to $C$ than the other. Depending on the color class of this endpoint, we say that the emerald or the violet endpoint of $\varepsilon$ \emph{faces} $C$.

\begin{lemma}\label{lem:arbo}
Let $G$ be a connected bipartite graph with root polytope $Q_G$. Let us fix an arbitrary triangulation $\mathscr T$ of $Q_G$. Let $\Sigma$ be a cycle-free subgraph of $G$ so that the corresponding simplex $\sigma\subset Q_G$ is an interior face of $\mathscr T$, and let $C$ be a connected component of $\Sigma$. Then there is a unique maximal simplex in $\mathscr T$ containing $\sigma$ so that in the corresponding spanning tree $\Gamma_C\supset\Sigma$, all edges of $\Gamma_C\setminus\Sigma$ have their violet endpoint face $C$.
\end{lemma}

Because of the resemblance  between $\Gamma_C$ and an arborescence in a directed graph, we may refer to this statement as the `arborescence lemma.'


\begin{proof}
We start with uniqueness. (This part of the proof works for all faces $\sigma$ of $\mathscr T$, including those along the boundary of $Q_G$.) Let us assume that $\Gamma_1$ and $\Gamma_2$ are different spanning trees in $G$ so that both satisfy the conditions stipulated for $\Gamma_C$. We will show that they violate the compatibility condition of Lemma \ref{lem:compa} and hence the corresponding simplices cannot both be part of $\mathscr T$. Let us pick an edge $\varepsilon\in\Gamma_1\setminus\Gamma_2$ and let $e$ be the emerald endpoint of $\varepsilon$. Let us also fix an arbitrary vertex $x$ of $C$ and consider the unique paths $p_1\subset\Gamma_1$ and $p_2\subset\Gamma_2$ from $e$ to $x$. The first edge along $p_1$ is clearly $\varepsilon$. Now let us find the first vertex $y$ (after $e$) along $p_1$ that is also a vertex along $p_2$ (since $x$ is such a vertex, $y$ is well-defined). If we follow $p_1$ from $e$ to $y$ and then follow $p_2$ from $y$ back to $e$, we obtain a cycle as described in \eqref{dva} of Lemma \ref{lem:compa}. Indeed, each time we step from an emerald to a violet vertex, that edge is either in $\Sigma$, hence in $\Gamma_1$, or if not then it has to be along the first half of our loop (from $e$ to $y$) which is again part of $\Gamma_1$. Similarly, steps taken from violet to emerald vertices are either in $\Sigma$, hence in $\Gamma_2$, or along the second half of the loop which is part of $\Gamma_2$ as well.

We will establish the existence of $\Gamma_C$ by an iterative argument that is very similar to the second half of the proof of Theorem 10.1 in \cite{hypertutte}. A spanning tree $\Gamma$ containing $\Sigma$ (or rather, the set of edges $\Gamma\setminus\Sigma$) can be viewed as a rooted tree in which the connected components of $\Sigma$ play the role of vertices and $C$ is the root. Let us call an edge of $\Gamma\setminus\Sigma$ \emph{bad} if its emerald endpoint faces $C$ and let us call other edges of $\Gamma\setminus\Sigma$ \emph{good}.

In order to be more precise, for any spanning tree $\Gamma$ containing $\Sigma$, let us consider the tree $\Gamma^{\mathrm{red}}$ which results from contracting each connected component of $\Sigma$ to a point (vertex). The edges of $\Gamma^{\mathrm{red}}$ inherit a good/bad classification from $\Gamma$. The vertex corresponding to $C$ will be treated as the root of $\Gamma^{\mathrm{red}}$. In a rooted tree, edges $\varepsilon$ have a well-defined distance~$d(\varepsilon)\in\N$ to the root (those adjacent to the root have $d=0$ etc.). Let us now associate the following quantities to $\Gamma$.
\begin{itemize}
\item Let $n(\Gamma)$ denote the smallest value of $d$ among bad edges of $\Gamma^{\mathrm{red}}$. If there are no bad edges, let $n(\Gamma)$ be one more than the maximal value of $d$ among the edges of $\Gamma^{\mathrm{red}}$.
\item For $1\le m\le n(\Gamma)$, let $\lambda_{\Gamma}(m)$ be the number of edges $\varepsilon$ of $\Gamma^{\mathrm{red}}$ with $d(\varepsilon)=m-1$. These values are positive. For $m>n(\Gamma)$, we let $\lambda_{\Gamma}(m)=0$. 
\end{itemize}
Then for a pair of spanning trees $\Gamma_1,\Gamma_2$ extending $\Sigma$, we write $\Gamma_1\prec\Gamma_2$ if either
\begin{enumerate}
\item\label{lexi} the sequence $\lambda_{\Gamma_1}(1),\lambda_{\Gamma_1}(2),\lambda_{\Gamma_1}(3)\ldots$ is smaller in lexicographic order than the sequence $\lambda_{\Gamma_2}(1),\lambda_{\Gamma_2}(2),\lambda_{\Gamma_2}(3)\ldots$, or
\item\label{tuske} the two sequences coincide (implying $n(\Gamma_1)=n(\Gamma_2)=n$) but the number of bad edges with $d=n$ is higher in $\Gamma_1^{\mathrm{red}}$ than in $\Gamma_2^{\mathrm{red}}$.
\end{enumerate}

The relation $\prec$ is a so called strict weak order on the set of spanning trees of $G$ containing $\Sigma$, in particular it is (obviously) transitive and asymmetric. To finish the proof of the Lemma, it suffices to show that if a tree $\Gamma\supset\Sigma$ corresponds to a maximal simplex in $\mathscr T$ and does not satisfy the conditions for $\Gamma_C$ (i.e., $\Gamma$ has bad edges), then there is another tree $\widetilde\Gamma$ containing $\Sigma$, with $\Gamma\prec\widetilde\Gamma$, which also appears as a maximal simplex in $\mathscr T$. Let us pick a bad edge $\varepsilon$ of $\Gamma$ from among those that are closest to the root in $\Gamma^{\mathrm{red}}$. As $\Sigma$ represents an interior face of $\mathscr T$, so does the larger set $\Gamma\setminus\{\varepsilon\}$. I.e., there exists an edge $\delta$ of $G$ so that $\widetilde\Gamma=(\Gamma\setminus\{\varepsilon\})\cup\{\delta\}$ corresponds to a maximal simplex of $\mathscr T$. An application of Lemma \ref{lem:compa} to $\Gamma$ and $\widetilde\Gamma$ shows that $\delta$ is a good edge for $\widetilde\Gamma$. 


It is easy to check that $\Gamma\prec\widetilde\Gamma$ is indeed true. If $\delta$ is closer to the root in $\widetilde\Gamma^{\mathrm{red}}$ than $\varepsilon$ was in $\Gamma^{\mathrm{red}}$, or if $\varepsilon$ was the only bad edge of $\Gamma^{\mathrm{red}}$ at distance $d(\varepsilon)$ from the root, then the relation holds by part \eqref{lexi} of its definition. Otherwise it holds by part \eqref{tuske}. This completes the proof.
\end{proof}

The following `path lemma' will soon be useful as well. The term `transfer of valence' was introduced right after Definition \ref{def:activity}.

\begin{lemma}\label{lem:path}
Let $\Gamma$ be a spanning tree of the connected bipartite graph $G$ that induces the hypertree $\mathbf f$ in $\mathscr H=(V,E)$ (here as usual, $E$ and $V$ are the color classes of $G$). Let $a,b,c$ be distinct elements of $E$ so that the unique path from $c$ to $a$ in $\Gamma$ passes through $b$. Suppose that $\mathbf f$ is such that $c$ can transfer valence to $a$. Then $\mathbf f$ is also such that $b$ can trans\-fer valence to $a$.
\end{lemma}

\begin{proof}
Assume that $b$ cannot transfer valence to $a$. Then there exists a set $T$ of hyperedges that is tight at $\mathbf f$, contains $a$, but does not contain $b$. (Cf.\ Lemma \ref{lem:ineq} and the discussion right after it.) We have to have $c\in T$ for otherwise no transfer would be possible from $c$ to $a$. For the same reason, in view of Lemma \ref{lem:tightcomp}, $a$ and $c$ have to be in the same connected component of the forest $\Gamma\big|_T$. Hence $a$ and $c$ can be connected by a path in $\Gamma\big|_T$, but since $b\not\in T$, that is different from the path, in $\Gamma$, that passes through $b$. This contradicts the cycle-freeness of $\Gamma$.
\end{proof}

Let $\mathscr T$ be a fixed triangulation of $Q_G$ and let us also fix an order of the color class $E$ (which we also think of as the set of hyperedges in the hypergraph $(V,E)$). 
Let $\gamma$ be a maximal simplex of $\mathscr T$ and $\Gamma$ the corresponding spanning tree of $G$. Let $\Gamma$ induce the hypertree $\mathbf f$ and let $e\in E$ be an internally inactive hyperedge with respect to $\mathbf f$. Let $e'$ be the smallest hyperedge so that $\mathbf f$ is such that $e$ can transfer valence to $e'$, and let the hypertree $\mathbf g$ be the result of that transfer. 

We will refer to the hyperedge $e'$ above as the \emph{favorite} of $e$ (with respect to the chosen order). If $e$ is indeed internally inactive then we have $e'<e$; for an internally active hyperedge, let us define its favorite to be itself. 


Let us now connect the two emerald markers $\mathbf f^+$ and $\mathbf g^+$ (cf.\ \eqref{eq:mark}) by a straight line segment, which we call the \emph{feeler} for the pair $(\mathbf f,e)$. (So technically, the feeler is defined to be a subset of the Minkowski sum $P_E$ discussed in Section~\ref{sec:root}, but we may also think of it as a subset of the cross-section $S_E\subset Q_G$ since $P_E$ and $S_E$ are related by a dilation.)
Let $\gamma'$ be the maximal simplex that the feeler enters right after it leaves $\gamma$. (In Lemma \ref{lem:feeler} we will see that the feeler exits $\gamma$ through a relative interior point of a facet so that it enters the interior of the next maximal simplex. It is possible for $\mathbf g^+$ to be the emerald marker of $\gamma'$, but it will typically not be the case.) If $\Gamma'$ is the spanning tree that corresponds to $\gamma'$, then the symmetric difference of $\Gamma$ and $\Gamma'$ consists of exactly two edges of $G$. In the next Lemma we identify one of those two edges.

\begin{lemma}\label{lem:feeler}
Let the hypertrees $\mathbf f,\mathbf g\in B_{\mathscr H}$ differ by a single transfer of valence from $e$ to $e'$. Let us connect the emerald markers $\mathbf f^+$ and $\mathbf g^+$ by a line segment $l$ in $P_E$. If $\mathbf f^+$ is an interior point of the Minkowski cell $M_\Gamma$ (for a spanning tree~$\Gamma$ inducing $\mathbf f$, cf.\ \eqref{eq:bigcell}) then $l$ leaves $M_\Gamma$ through a codimension $1$ stratum of its boundary which corresponds to removing from $\Gamma$ the first edge along the unique path from $e$ to $e'$.
\end{lemma}

\begin{proof}
In Lemma \ref{lem:coordinates} we described the starting point $\mathbf f^+$ of $l$ in terms of the direct sum decomposition \eqref{eq:bigcell} of $M_\Gamma$. We will extend that description to the initial segment of $l$. Let the path $p$ in $\Gamma$ from $e$ to $e'$ be $\varepsilon_1,\delta_1,\varepsilon_2,\delta_2,\ldots,\varepsilon_k,\delta_k$ and assume that at $\mathbf f^+$ these edges have the weights $u_1,v_1,u_2,v_2,\ldots,u_k,v_k$, respectively, as given in Lemma \ref{lem:coordinates}. (Note that emerald and violet vertices alternate along $p$ and since $e$ and $e'$ are both emerald, there has to be an even number of edges.) Then we claim that $l$ is parametrized by the weights
\begin{equation}\label{eq:feeler}
u_1-t,v_1+t,u_2-t,v_2+t,\ldots,u_k-t,v_k+t
\end{equation}
along $p$, while all other weights are constant.
Indeed, this ensures that sums of weights on edges adjacent to a violet vertex remain fixed at $1$; the sums of weights on edges adjacent to an emerald vertex also remain constant except in the cases of $e$ and $e'$, which is consistent with the direction of $l$.

The segment $l$ reaches the boundary of $\gamma$ (which is the maximal simplex corresponding to $\Gamma$) when it reaches the boundary of its (dilated) cross-section $M_\Gamma$, i.e., when one of the weights becomes zero. By \eqref{eq:feeler}, this will occur on one of the $\varepsilon_i$. In fact, from Lemma \ref{lem:coordinates} it easily follows that $u_1<u_2<\cdots<u_k$, from which we see that it is the edge $\varepsilon_1$ (and only that) that gets removed from $\Gamma$ when we pass from $\gamma$ to the adjacent maximal simplex. 
\end{proof}

Note that Lemma \ref{lem:feeler} does not provide any information on the unique edge in $\Gamma'\setminus\Gamma$. That depends on the nature of the triangulation $\mathscr T$. If the feeler intersects only the maximal simplices marked by $\mathbf f^+$ and $\mathbf g^+$, then this edge will be adjacent to $e'$; otherwise we have basically no control over it.

\begin{rmk}\label{rem:rearrange}
Let us substitute $t=u_1$ in \eqref{eq:feeler} and use the resulting values as new weights on the edges $\varepsilon_1,\delta_1,\varepsilon_2,\delta_2,\ldots,\varepsilon_k,\delta_k$, while we keep the weights provided by $\mathbf f^+$ on other edges of $\Gamma$. These are the barycentric coordinates (in the simplex $\gamma$ that corresponds to $\Gamma$) of the point where the feeler leaves the Minkowski cell, i.e., the simplex. But they also correspond to the following operation. Let $C$ be the component of $\Gamma\setminus\{\varepsilon_1\}$ that does not contain $e'$. Now remove $\varepsilon_1$ from $\Gamma$ and then identify $e$ and $e'$ to obtain a new tree. (One can imagine this as transporting $C$ along the path $p$ from $e$ to $e'$.) Let the merged vertex have multiplicity $2$. If we apply the edge-weight formula of Lemma \ref{lem:coordinates} to this tree, it is easy to see that we obtain exactly the weights we have just described. This observation may sound ad hoc but it will be useful in the proof of Theorem \ref{thm:tech} below.
\end{rmk}

The last result in this section is an extension of \cite[Lemma 7.4]{hypertutte}. It allows us to generate hypercubes of hypertrees in certain situations. We state it in the hypergraph context even though it holds for integer polymatroids as well.

\begin{lemma}\label{lem:cube}
Let $\{\,a_1,\ldots,a_m\,\}$ and $\{\,x_1,\ldots,x_m\,\}$ be disjoint collections of hyperedges in the connected hypergraph $\mathscr H=(V,E)$. Let $\mathbf f\colon E\to\N$ be a hypertree which is such that $x_j$ can transfer valence to $a_j$ for all $j$ but $x_j$ cannot transfer valence to $a_i$ for any $i<j$. Then $\mathbf f$ is such that any subset of the transfers mentioned above is simultaneously possible, that is, for any subset $J\subset\{\,1,\ldots,m\,\}$, the vector $\mathbf f+\sum_{j\in J}(\mathbf i_{\{a_j\}}-\mathbf i_{\{x_j\}})$ is another hypertree.
\end{lemma}

\begin{proof}
Without loss of generality we may assume that $J=\{\,1,\ldots,m\,\}$ and $m\ge2$. If the statement fails to be true then there exists a smallest index $k$ so that $\mathbf u=\mathbf f+\sum_{j=1}^{k-1}(\mathbf i_{\{a_j\}}-\mathbf i_{\{x_j\}})$ is a hypertree but $\mathbf f+\sum_{j=1}^{k}(\mathbf i_{\{a_j\}}-\mathbf i_{\{x_j\}})$ is not. The latter implies that there is a set $K$ of hyperedges that is tight at $\mathbf u$, contains $a_k$ but does not contain $x_k$. Now at $\mathbf f$, there exist tight sets $L_1,\ldots,L_{k-1}$ of hyperedges so that $L_i$ contains $a_i$ but does not contain $x_k$. By Lemma \ref{lem:tight}, $L=L_1\cup\cdots\cup L_{k-1}$ is also tight at $\mathbf f$. Since $a_1,\ldots,a_{k-1}\in L$, we have to have $x_1,\ldots,x_{k-1}\in L$ too for the first $k-1$ of our assumed transfers to be possible. That implies that $L$ is also a tight set at $\mathbf u$ and hence the same is true for $L\cup K$. But since $\mathbf u$ and $\mathbf f$ produce the same sum of values over elements of $L\cup K$, we eventually obtain that $L\cup K$ is a tight set at $\mathbf f$. That is a contradiction because $L\cup K$ separates $x_k$ from $a_k$.
\end{proof}

\section{The main result}\label{sec:main}


To prove Theorem \ref{thm:main}, it remains to show that the interior polynomial $I$ of the connected hypergraph $(V,E)$ satisfies
\begin{equation}\label{eq:goal}
I(x)=a_0+a_1x+\cdots+a_dx^d.
\end{equation} 
First we give the reduction of this claim to a somewhat technical result, Theorem~\ref{thm:tech}, and then proceed to prove the latter.

The sequence $a_0,a_1,\ldots,a_d$ of \eqref{eq:goal}, where first of all $d=|E|+|V|-2$ is the dimension of the root polytope $Q_G$, was defined in \eqref{eq:ehrhart} (and then again in \eqref{eq:again}) in terms of the Ehrhart polynomial of $Q_G$. However, in Theorem \ref{thm:h=E} we have already equated $\sum_{k=0}^da_kx^k$ to the right hand side of \eqref{eq:I=h} so that from now on we may think of $a_0,a_1,\ldots,a_d$ as the coefficient sequence of the $h$-vector $h_G$. In particular, as we have seen in Remark \ref{rem:shell}, if a triangulation of $Q_G$ has a shelling order then $a_k$ is the number of maximal simplices that are adjacent (through a facet) to exactly $k$ `previous' maximal simplices.


Let us fix a shellable triangulation $\mathscr T$ (such as a regular triangulation\footnote{Roughly speaking, a regular triangulation is one obtained by assigning a (generic) new coordinate to each vertex, thereby lifting them to one dimension higher, then taking their convex hull and projecting `back down' its upper boundary. Not all triangulations of a root polytope are regular, not even for 
complete bipartite graphs \cite{san}.}, cf.\ \cite{swartz})
of $Q_G$ with $f$-vector $f(y)$ as in \eqref{eq:f}. Let the number of interior simplices of $\mathscr T$ of dimension $m$ be $\tilde f_m$ for $m=0,1,\ldots,d$. (Note that $\tilde f_0=0$ and $\tilde f_d=f_d$.) Similar to the proof of Theorem~\ref{thm:h=E} (and especially to Remark \ref{rem:shell}), we may use $\mathscr T$ in conjunction with Lemma \ref{lem:unimod} to count \emph{interior} lattice points in the dilation $s\cdot Q_G$ (where $s$ is a positive integer). This time we will count points along interior facets when the \emph{second} maximal simplex adjacent to the facet is encountered in the shelling order. By doing so we find that the number of such points is
\[\tilde\varepsilon_G(s)=a_0{s-1 \choose d}+a_1{s \choose d}
+\cdots+a_k{s-1+k \choose d}+\cdots+a_d{s-1+d \choose d}.\]
This is because of the following reason: in the interior of each maximal simplex there are ${s-1 \choose d}$ lattice points (cf.\ Remark \ref{rem:ehr}) whose convex hull is a $d$-dimensional simplex of sidelength $s-d-1$. (We may assume here that $s\ge d+1$ since from Ehrhart theory we know that the result of our lattice point count is a polynomial in $s$, determined uniquely by just finitely many of its values.) If the interior (to $s\cdot Q_G$) lattice points along $k$ facets of the maximal simplex are to be counted as well, the sidelength of the convex hull (which is still a simplex) increases by $k$ so that the number of lattice points becomes ${s-1+k \choose d}$.


On the other hand, using the obvious partition of the interior of $Q_G$ and Remark~\ref{rem:ehr}, the same quantity can also be expressed as 
\[\tilde\varepsilon_G(s)=\tilde f_d{s-1 \choose d}+\tilde f_{d-1}{s-1 \choose d-1}+\cdots+\tilde f_{d-l}{s-1 \choose d-l}+\cdots+\tilde f_1{s-1 \choose 1}.\]

Let us manipulate the first expression using Vandermonde's identity\footnote{This step is also found in one of Andrews's proofs \cite{andrews} of the Saalsch\"utz formula.}
\[{s-1+k \choose d}={k \choose 0}{s-1 \choose d}+{k \choose 1}{s-1 \choose d-1}+\cdots+{k \choose l}{s-1 \choose d-l}+\cdots+{k \choose k}{s-1 \choose d-k}\]
and equate coefficients of ${s-1 \choose d-l}$ (noting that these polynomials have different degrees in $s$ and hence they are linearly independent) to find that 
\[\tilde f_{d-l}={l\choose l}a_l+{l+1\choose l}a_{l+1}+\cdots+{d \choose l}a_d.\]
Notice that the right hand side is exactly the coefficient of $(x-1)^l$ in the Taylor expansion, centered at $1$, of $\sum_{k=0}^da_kx^k$. Therefore to show \eqref{eq:goal}, it suffices to prove that $I(x)$ has the same expansion, that is that the coefficient of $(x-1)^k$ in the Taylor expansion of $I(x)$ centered at $1$ is $\tilde f_{d-k}$ for all $k\ge0$ (here for negative $u$ we define $\tilde f_u=0$). In other words, what is left to prove is that
\begin{equation}
I^{(k)}(1)=k!\cdot\tilde f_{d-k}
\end{equation}
holds for all $k\ge0$. We note that the case when $k=0$ is settled in \cite{alex}.

Now if we recall Definition \ref{def:I} of the interior polynomial, we see that the Taylor coefficient in question is $\sum_{i=k}^{|E|-1}{i \choose k}\chi_{i}$, where $\chi_i$ is the number of those hypertrees in $(V,E)$ whose internal inactivity is exactly $i$. Put another way, it is the number of tuples consisting of a hypertree and exactly $k$ of its internally inactive hyperedges. Therefore Theorem \ref{thm:main} is a corollary of the following statement. (Note that from here on we will not rely on the condition of shellability.)

\begin{thm}\label{thm:tech}
Let $\mathscr T$ be a triangulation of the root polytope $Q_G$ of the connected bipartite graph $G$, and let $k\ge0$ be an integer. Then the number of codimension~$k$ interior simplices of \ $\mathscr T$ agrees with the number of pairs $(\mathbf f,S)$, where $\mathbf f$ is a hypertree in the hypergraph $(V,E)$ induced by $G$ and $S$ is a set of $k$ distinct internally inactive hyperedges with respect to $\mathbf f$. Here inactivity is defined with respect to an arbitrarily fixed order on $E$.
\end{thm}

\begin{proof}
As the case $k=0$ follows from \cite[Lemma 12.8 and Theorem 12.9]{alex}, we may assume that $k\ge1$.
We shall define a map $\sigma$ from the second named set to the first and show that it is one-to-one and onto.

If the pair $(\mathbf f,\{\,e_1,\ldots,e_k\})$ is as described in the Theorem, then first we find the unique maximal simplex $\gamma_{\mathbf f}$ in $\mathscr T$ which belongs to a spanning tree $\Gamma_{\mathbf f}$ of $G$ that induces $\mathbf f$. (See \cite{alex} or subsection \ref{ssec:cross}.) For each $1\le i\le k$, we construct the feeler for the pair $(\mathbf f,e_i)$. By Lemma~\ref{lem:feeler}, the $i$'th feeler leaves $\gamma_{\mathbf f}$ through a facet that corresponds to (removing) an edge $\varepsilon_i$ of $\Gamma_{\mathbf f}$ which is adjacent to $e_i$. Hence these $k$ facets are mutually different and thus their intersection is a codimension $k$ face $\sigma(\mathbf f,\{\,e_1,\ldots,e_k\})$ of $\mathscr T$.

\smallskip

\emph{Well-definedness of $\sigma$.} We need to show that the face $\sigma(\mathbf f,\{\,e_1,\ldots,e_k\})$ lies in the interior of $Q_G$. This is because one is able to find enough hypertrees near $\mathbf f$ (namely the endpoints of the $k$ feelers and some others) so that the convex hull of the corresponding emerald markers, which is $k$-dimensional 
and interior to $Q_G$, meets $\sigma(\mathbf f,\{\,e_1,\ldots,e_k\})$ in just one point and that point is in the relative interior of both sets. A detailed computation follows.

Let the favorite of $e_i$ be $e_i'$ ($1\le i\le k$). (Favorites were defined in Section \ref{sec:prep}.) The sets $\{\,e_1,\ldots,e_k\,\}$ and $\{\,e_1',\ldots,e_k'\,\}$ are disjoint by \cite[Lemma 5.2]{hypertutte}, which says that if a hypertree is such that the hyperedge $c$ can transfer valence to $b$ and $b$ can transfer valence to $a$, then $c$ can transfer valence to $a$ as well. As $\{\,e_1',\ldots,e_k'\,\}$ may have less than $k$ elements, let us also write it as $S'=\{\,a_1,a_2,\ldots,a_m\,\}$ where we assume $a_1<a_2<\cdots<a_m$. There is then an obvious partition of $S=\{\,e_1,\ldots,e_k\,\}$ into $m$ parts $S_1,\ldots,S_m$ so that the favorite of each element of $S_i$ is $a_i$.

For each $i=1,\ldots,m$, we may write an arbitrary linear combination $\mathbf v_i$ of the vectors $\mathbf i_{\{a_i\}}-\mathbf i_{\{e\}}$ for $e\in S_i$ so that the coefficients are positive 
and their sum is less than $1$. Then, by Lemma \ref{lem:cube}, we may add the sum of these vectors to the emerald marker $\mathbf f^+$ (cf.\ \eqref{eq:mark}) and obtain an interior point of $Q_G$ that way. 
It is enough to show that $\sigma(\mathbf f,\{\,e_1,\ldots,e_k\})$ contains one of these points.

Since for all $i=1,\ldots,m$, the hypertree $\mathbf f$ is such that no hyperedge from $K_i=S_{i+1}\cup\cdots\cup S_m$ can transfer valence to $a_i$, there exist sets of hyperedges that are tight at $\mathbf f$ and separate $a_i$ from the elements of $K_i$. The intersection $T_i'$ of these sets is also tight at $\mathbf f$ (by Lemma \ref{lem:tight}), contains $a_i$ (hence it contains $S_i$, too), and is disjoint from $K_i$. Lemmas \ref{lem:path} and \ref{lem:tightcomp} allow us to assume that $\Gamma_{\mathbf f}\big|_{T_i'}$ is connected: indeed, all hyperedges along paths in $\Gamma_{\mathbf f}$ between elements of $S_i$ and $a_i$ are such that they can transfer valence to $a_i$ and thus they have to be in $T_i'$. By letting $T_i=T_1'\cup\cdots\cup T_i'$, we get a nested sequence $T_1\subset\cdots\subset T_m$ of sets, each of which is tight (again by Lemma \ref{lem:tight}) at $\mathbf f$. Let us reiterate that 
\begin{equation}\label{eq:nested}
\text{for each }e_j\in S\text{, the path }p_{e_j}\text{ in }\Gamma_{\mathbf f}\text{ from }e_j\text{ to }e_j'\text{ lies in }\Gamma_{\mathbf f}\big|_{T_i}, 
\end{equation}
where $i$ is the index so that $e_j'=a_i$, i.e., that $e_j\in S_i$. 

We will construct the vectors $\mathbf v_i$ mentioned above inductively, starting from $\mathbf v_m$. (One way to think about this next part of the proof is that we will build a polygonal path in $P_E$ from $\mathbf f^+$ to $\sigma(\mathbf f,\{\,e_1,\ldots,e_k\})$, that is from an interior point of the Minkowski cell $M_{\Gamma_{\mathbf f}}$ of \eqref{eq:bigcell} to a codimension $k$ stratum of its boundary. We will choose the length of each segment carefully so that the path stays within the cell.)
The set $S_m$ has a partial order $<$ defined by the rooted tree $(\Gamma_{\mathbf f},a_m)$, namely we let $x<y$ if the path in $\Gamma_{\mathbf f}$ from $y$ to $a_m$ passes through $x$. The coefficients of the vectors $\mathbf i_{\{a_m\}}-\mathbf i_{\{e\}}$ in $\mathbf v_m$ (where $e\in S_m$) are again defined inductively starting from the maximal elements of $S_m$. If $e$ is one of those and the weight, in the de\-scrip\-tion of $\mathbf f^+$ provided by Lemma \ref{lem:coordinates}, is $u$ on the edge $\varepsilon$ that starts the path $p_e$ from $e$ to $a_m$ in $\Gamma_{\mathbf f}$, then let the coefficient of $\mathbf i_{\{a_m\}}-\mathbf i_{\{e\}}$ be $u$. Because (by the definition of the partial order and the observation \eqref{eq:nested} above) none of the paths $p_{e'}$ for $e\ne e'\in S$ passes through $e$, in particular none of them contains $\varepsilon$, this is the only time in the process that the weight on $\varepsilon$ changes (cf.\ \eqref{eq:feeler} in the proof of Lemma \ref{lem:feeler}). Hence our choice $u$ for the coefficient of $\mathbf i_{\{a_m\}}-\mathbf i_{\{e\}}$ is the unique way we can reduce the weight on $\varepsilon$ to $0$.

As discussed in the proof of Lemma \ref{lem:feeler}, edge weights (barycentric coordinates) representing the point $\mathbf f^++u(\mathbf i_{\{a_m\}}-\mathbf i_{\{e\}})$ of $\gamma_{\mathbf f}$ are positive except for that of $\varepsilon$. As explained in Remark \ref{rem:rearrange}, this is `demonstrated' by applying the formula of Lemma~\ref{lem:coordinates} to a rearranged tree, where we erase $\varepsilon$ and transport the part $C$ of $\Gamma_{\mathbf f}$ that `lies beyond' $e$ to the favorite $a_m$ of $e$. Since paths $p_{e'}$ for $e\ne e'\in S$ (along which weight changes are due to occur) are either disjoint from, or contained in $C$, this transporting operation does not disrupt the structure of the tree that is relevant for the rest of the construction.

Next, we iterate the process that was described in the previous two paragraphs. There will be one step for each element $e$ of $S_m$ and it will occur after all elements of $S_m$ above $e$ in the partial order have been dealt with. During the step we determine the coefficient $u_e$ of the vector $\mathbf i_{\{a_m\}}-\mathbf i_{\{e\}}$ in the linear combination $\mathbf v_m$, namely we define $u_e$ to be the current weight on the edge $\varepsilon_e$ which is adjacent to $e$ and points toward $a_m$. In light of \eqref{eq:feeler} and the fact that $\varepsilon_e$ is not part of the path $p_h$ for any element $h\in S$ that we will address later (which is, again, guaranteed by \eqref{eq:nested} and the partial order), this is the unique choice that results in the weight on $\varepsilon_e$ becoming $0$. At the end of each step we update the weights on the edges of $\Gamma_{\mathbf f}$ using \eqref{eq:feeler}, or rather on the edges of the tree that results from our previous transports, and note that these new weights are the same as what the recipe of Lemma \ref{lem:coordinates} produces for the rearranged tree where we erase $\varepsilon_e$ and identify $e$ with $a_m$, denoting the new point by $a_m$ and increasing its multiplicity by $1$. In particular, all edge weights in the rearranged tree are positive.

After we have assigned coefficients to each vector $\mathbf i_{\{a_m\}}-\mathbf i_{\{e\}}$ where $e\in S_m$, noting that they are all positive, we also have to make sure that the sum of the coefficients is less than $1$. That is because it is equal to the proportion, among all elements of $E$, of the emerald vertices of all the (disjoint) subtrees that have been transported to $a_m$ so far in the process. This set includes the vertices that were merged with $a_m$ but it does not include the `original' vertex $a_m$ itself, and hence its proportion in $E$ is less than $1$.

After this we continue performing the same operations as above (to the rearranged tree and the updated weights) for $a_{m-1}$, $a_{m-2}$ and so forth down to $a_1$. By the exact same arguments as above, each vector $\mathbf i_{\{a_i\}}-\mathbf i_{\{e\}}$, for $e\in S_i$, will receive a positive coefficient in the linear combination $\mathbf v_i$. The sum of the coefficients, for each given $i$, is strictly less than $1$. These coefficients are also the unique choices that result in a weight of $0$ on each edge $\varepsilon_i$, $1\le i\le k$, i.e., which guarantee that $\mathbf f^++\sum_{i=1}^k\mathbf v_i$ is a point of the simplex $\sigma(\mathbf f,S)$. We also see that it is an interior point of said simplex because the weights on edges other than the $\varepsilon_i$ remain positive at the end. This completes the proof of well-definedness for our correspondence $\sigma$.

\smallskip

Let now $\sigma$ be a codimension $k$ interior face of $\mathscr T$. Let the edges of $G$ which correspond to the vertices of $\sigma$ form the $(k+1)$-component forest $\Sigma$. From the point of view of $\Sigma$, the statement that $\sigma=\sigma(\mathbf f,\{\,e_1,\ldots,e_k\})$ for some pair $(\mathbf f,\{\,e_1,\ldots,e_k\})$ can be described as follows. We add $k$ distinct edges $\varepsilon_1,\ldots,\varepsilon_k$ of $G$ to $\Sigma$ to obtain a spanning tree $\Gamma$ which is a realization of $\mathbf f$, and which is also such that the corresponding maximal simplex is in $\mathscr T$. For each $i$, the emerald endpoint of $\varepsilon_i$ is $e_i$; furthermore and most crucially, $\varepsilon_i$ is the first edge along the unique path in $\Gamma$ that connects $e_i$ to its favorite (in relation to $\mathbf f$) hyperedge. In particular, the emerald endpoints of the `extra' edges $\varepsilon_1,\ldots,\varepsilon_k$ are all distinct and, obviously, $\Sigma$ and the extra edges determine the pair $(\mathbf f,\{\,e_1,\ldots,e_k\})$. Let us call the collection $\varepsilon_1,\ldots,\varepsilon_k$ \emph{good} with respect to $\sigma$ if it fits the description above (with the given triangulation~$\mathscr T$ and for the pair $(\mathbf f,\{\,e_1,\ldots,e_k\})$ that the collection determines).

We need to show that for the given $\sigma$, there exists a unique collection $\varepsilon_1,\ldots,\varepsilon_k$ of edges not in $\Sigma$ that is good in the sense of the previous paragraph. We start with uniqueness.

\smallskip

\emph{Injectivity of $\sigma$.}
Let $\varepsilon_1,\ldots,\varepsilon_k$ and $\delta_1,\ldots,\delta_k$ be two different good collections. If we enlarge $\Sigma$ with the edges common to them, their remaining parts are still good collections for this larger forest/interior cell. (Note that favorites of hyperedges are decided not by $\Sigma$ but by the actual spanning tree containing it and the latter remains the same even as we declare $\Sigma$ to be bigger.) Hence we may assume that the two collections are disjoint. 

Let the hyperedge $a$ be the smallest favorite that occurs among simplices of $\mathscr T$ containing $\sigma$. (By this we mean the favorites of those hyperedges that received an extra edge in the extension and became internally inactive with respect to the hypertree induced by the extension. That set is non-empty because $k\ge1$ and the emerald endpoint of $\varepsilon_1$ has a favorite different from itself.) 
Let $\Gamma_a$ be the spanning tree in $G$ provided by Lemma~\ref{lem:arbo} applied to $\Sigma$ and its component containing $a$. That is, $\Gamma_a$ represents a maximal simplex in $\mathscr T$, contains $\Sigma$, and has all its edges beyond $\Sigma$ face $a$ with their violet endpoint. We are going to argue that $\Gamma_a$ and the spanning trees $\Gamma^\varepsilon=\Sigma\cup\{\,\varepsilon_1,\ldots,\varepsilon_k\,\}$ and $\Gamma^\delta=\Sigma\cup\{\,\delta_1,\ldots,\delta_k\,\}$ cannot all be compatible (i.e., it is not possible for all three pairs to satisfy the condition given in Lemma \ref{lem:compa}), which will contradict their coexistence in $\mathscr T$.

Let the emerald endpoints of the $\varepsilon_i$ form the set $S^\varepsilon\subset E$, and let the emerald endpoints of the $\delta_j$ form the set $S^\delta\subset E$. Within these, let $S^\varepsilon_a$ and $S^\delta_a$ denote the set of those hyperedges whose favorite is $a$. If $e\in S^\varepsilon_a$, then by Lemma \ref{lem:path}, all elements of $S^\varepsilon$ along the path $p_e$ in $\Gamma^\varepsilon$ from $e$ to $a$ also belong to $S^\varepsilon_a$, and hence the edges $\varepsilon_i$ adjacent to them belong to $p_e$ so that their violet endpoints face $a$. Similar observations hold true for $S^\delta_a$. If a component of $\Sigma$ contains an element of $S^\varepsilon_a$ as well as an element of $S^\delta_a$ then, just like in the proof of Lemma \ref{lem:arbo}, it is easy to see that $\Gamma^\varepsilon$ and $\Gamma^\delta$ are not compatible trees. In the same way we obtain that 
\begin{multline}\label{eq:text}
\text{the edges }\varepsilon_i\text{ adjacent to elements of }S^\varepsilon_a\text{, as well as the edges }\delta_j\\
\text{adjacent to elements of }S^\delta_a\text{, are edges in }\Gamma_a\text{, too.}
\end{multline}

Let now $\mathbf f^\varepsilon$ be the hypertree induced by $\Gamma^\varepsilon$. For each element $e\in S^\varepsilon\setminus S^\varepsilon_a$, there is a set of hyperedges that is tight at $\mathbf f^\varepsilon$, contains $a$, and does not contain $e$. The intersection $T^\varepsilon$ of these sets is also tight by Lemma \ref{lem:tight}, it is disjoint from $S^\varepsilon\setminus S^\varepsilon_a$, and it contains $S^\varepsilon_a$ for otherwise those hyperedges could not transfer valence to $a$. By the observation \eqref{eq:text}, we see that $\sum_{e\in T^\varepsilon}\mathbf f^\varepsilon(e)\le\sum_{e\in T^\varepsilon}\mathbf f_a(e)$ but since $\sum_{e\in T^\varepsilon}\mathbf f^\varepsilon(e)$ is already at the maximal value $\mu(T^\varepsilon)$ allowed for any hypertree, it follows that
\begin{itemize}[leftmargin=10pt]
\item $T^\varepsilon$ is also a tight set at the hypertree $\mathbf f_a$ induced by $\Gamma_a$ and
\item $\Gamma_a\big|_{T^\varepsilon}=\Gamma^\varepsilon\big|_{T^\varepsilon}$ consists only of edges of $\Sigma$ and the $\varepsilon_i$ adjacent to elements of $S^\varepsilon_a$.
\end{itemize}
By the same token there is a set $T^\delta\subset E$ that is tight at $\mathbf f_a$, contains $a$, and intersects $S^\delta$ in exactly $S^\delta_a$. By Lemma \ref{lem:tight}, $T=T^\varepsilon\cap T^\delta$ is also tight at $\mathbf f_a$. It contains $a$ and since $S^\varepsilon_a\cap S^\delta_a=\varnothing$, the forest $\Gamma_a\big|_T$ consists only of edges of $\Sigma$. However that means that $a$ cannot receive a transfer of valence at any hypertree induced by an extension of $\Sigma$, which contradicts the definition of $a$.

\smallskip

\emph{Surjectivity of $\sigma$.}
It remains to show that any $(k+1)$-component forest $\Sigma\subset G$, provided that it represents an interior face of $\mathscr T$, can be extended to a spanning tree by a good collection of edges so that the corresponding maximal simplex occurs in $\mathscr T$. It is tempting to think that the tree $\Gamma_a$ defined above (i.e., the set of its edges beyond $\Sigma$) would satisfy the requirement but, unfortunately, that is not always the case (see Example~\ref{ex:badcase} below). Instead, we devise the following iterative procedure.

Let $a_1=a$ be the `smallest favorite' hyperedge and $\Gamma_1=\Gamma_a$ the spanning tree that we used before. 
In the set $S_1$ of the $k$ emerald endpoints of the edges in $\Gamma_1\setminus\Sigma$, let us consider those that cannot transfer valence to $a_1$. These are separated from $a_1$ by sets that are tight at the hypertree $\mathbf f_1=\mathbf f_a$ induced by $\Gamma_1$. The intersection $T_1$ of these sets is also tight at $\mathbf f_1$, it contains $a_1$, and hence it intersects $S_1$ in exactly those of its elements that can transfer valence to $a_1$ at $\mathbf f_1$. By Lemmas \ref{lem:path} and 
\ref{lem:tightcomp}, we may assume that $\Gamma_1\big|_{T_1}$ is connected. 
We add to $\Sigma$ the edges of $\Gamma_1\setminus\Sigma$ that are adjacent to elements of $S_1\cap T_1$. During the rest of the construction we will keep the subgraph $\Gamma_1\big|_{T_1}$ fixed so that $T_1$ remains a tight set at subsequent hypertrees, too. This serves to `protect' the transfers of valence that we have found from elements of $S_1\cap T_1$ to $a_1$, as well as to maintain the situation that the extra edges adjacent to those emerald vertices point toward $a_1$.
Indeed, if an element of $S_1\cap T_1$ got separated from $a_1$ by a tight set at a later stage, then the intersection of that set with $T_1$ would also be tight at that hypertree; but then the same set would be tight at $\mathbf f_1$ as well, contradicting the possibility of the transfer at $\mathbf f_1$.

Next, let $a_{2}$ be the smallest hyperedge not in $T_1$. Let us apply Lemma \ref{lem:arbo} to the forest $\Sigma\cup(\Gamma_1\big|_{T_1})$ and its component containing $a_{2}$ to obtain the spanning tree~$\Gamma_2$ that induces the hypertree $\mathbf f_2$. There is again a tight set $T'_2$ of hyperedges that contains $a_2$ and all emerald endpoints of the edges in $\Gamma_2\setminus\Sigma\setminus(\Gamma_1\big|_{T_1})$ that can transfer valence to $a_2$ at $\mathbf f_2$ but does not contain any of those that cannot. We may assume that $\Gamma_2\big|_{T'_2}$ is connected. We let $T_2=T_1\cup T'_2$; by Lemma \ref{lem:tight}, this set is also tight at $\mathbf f_2$, as well as at later hypertrees. The set $T_2$ will protect the transfers of valence that we have just found to $a_2$. The set $T_1$, on the other hand, will guarantee that the givers of those transfers will not find receivers smaller than $a_2$ at any time during the rest of the process.

We iterate our construction, always taking the smallest hyperedge $a_{i+1}$ outside of the last element $T_i$ of our sequence of nested tight sets. In each step $a_{i+1}$, possibly with other hyperedges, will be added to $T_i$ to form the next tight set $T_{i+1}$. Thus the sequence $\{T_i\}$ is strictly increasing, and eventually the tight set will be $E$ itself. At that stage, all emerald vertices that have an extra edge adjacent to them have it point in the direction of their favorite. This completes the proof of surjectivity and hence the proof of the Theorem, as well as the proof of Theorem \ref{thm:main}.
\end{proof}

\begin{ex}\label{ex:badcase}
We use the plane bipartite graph $G$ of Example \ref{ex:12a1097} to illustrate some of the subtlety of the proof of Theorem \ref{thm:tech}. Let us consider one of the tri\-angul\-ations provided by \cite[Theorem~1.1]{homfly}. Namely, we orient the edges of the dual graph $G^*$ so that each has an emerald point to its right, we fix the root of $G^*$ in the outside region, construct all (say, outgoing) spanning arborescences relative to it, and then consider the spanning trees of $G$ that are dual to those. The maximal simplices of $Q_G$ that correspond to the latter form a triangulation $\mathscr T$.

\begin{figure}[htbp]
\labellist
\small
\pinlabel $a$ at 900 560
\pinlabel $b$ at 1170 830
\pinlabel $c$ at 830 1170
\pinlabel $d$ at 850 950
\pinlabel $\Gamma_1$ at -80 270
\pinlabel $\Gamma_2$ at 1830 270
\pinlabel $\Gamma_3$ at 1830 1480
\pinlabel $\Gamma_4$ at -80 1480
\endlabellist
   \centering
   \includegraphics[width=2in]{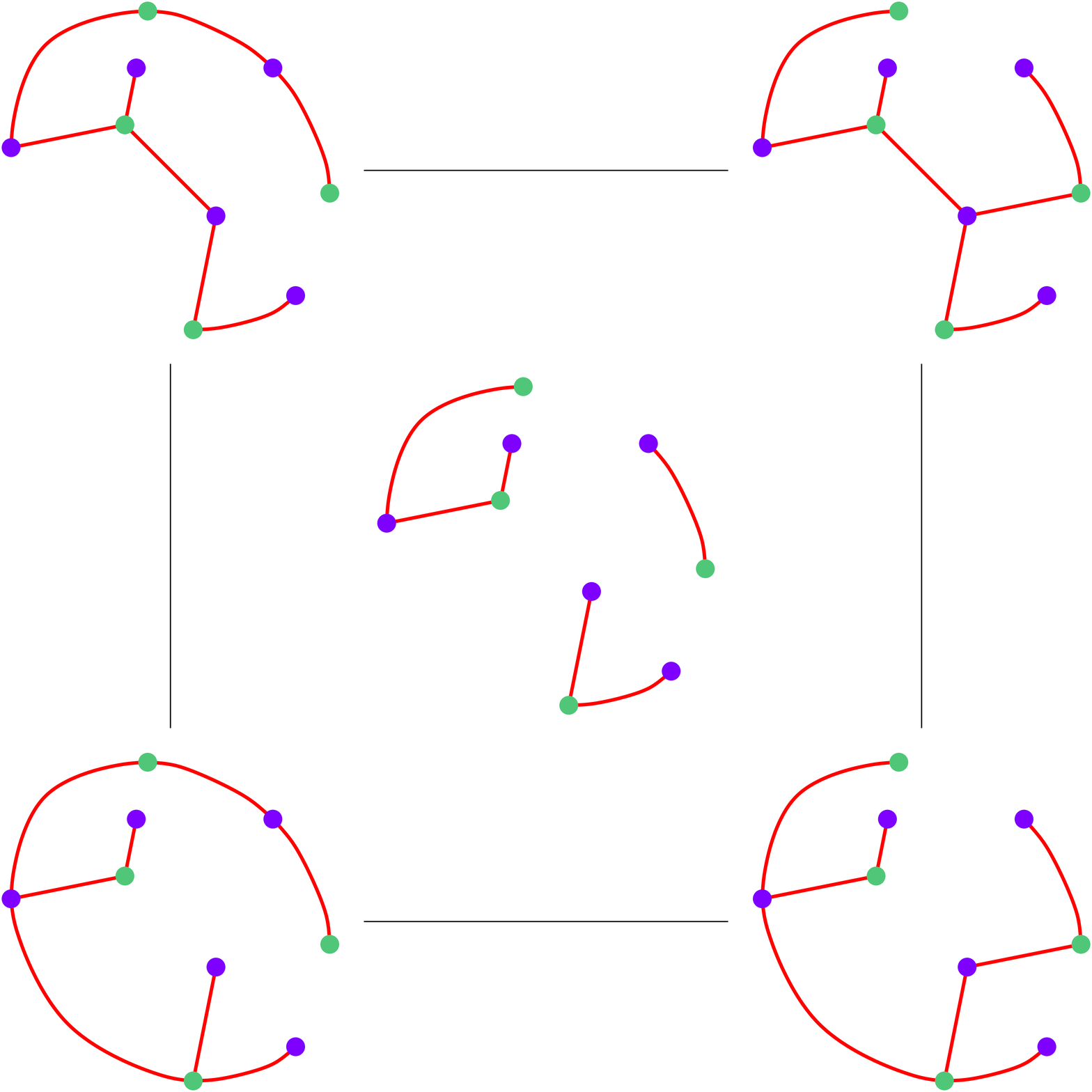} 
   \caption{A codimension $2$ interior simplex and the four maximal simplices containing it, each represented by the corresponding cycle-free subgraph.}
   \label{fig:badcase}
\end{figure}

We let $\Sigma$ be the three-component forest shown in the middle of Figure \ref{fig:badcase}. The corresponding codimension $2$ simplex $\sigma$ lies in the interior of $Q_G$ and it is contained in exactly four maximal simplices. These are represented by the corresponding spanning trees $\Gamma_i$, $1\le i\le4$. The four black line segments indicate adjacency through a facet. Let us also use the symbol $\mathbf f_i$ to denote the hypertree (in the hypergraph with emerald hyperedges) induced by $\Gamma_i$.

Let us choose the order $a<b<c<d$ on the emerald color class $E$. First we illustrate the last (surjectivity) part of the proof. The hypertree $\mathbf f_2$ is the one that extends the valence distribution of $\Sigma$ at the lexicographically smallest multiset, namely $\{\,a,b\,\}$. That hypertree does not in general have to be realized by a maximal simplex adjacent to $\Sigma$ and in any case, only $d$ is internally inactive with respect to $\mathbf f_2$ and $d$ did not even receive an extra edge when $\Sigma$ was extended to $\Gamma_2$. The spanning tree that results if we apply Lemma \ref{lem:arbo} to the component of $\Sigma$ containing $a$ is $\Gamma_3$. The only internally inactive hyperedge with respect to $\mathbf f_3$ is $d$. Indeed, the set $\{\,a,d\,\}$ is tight at $\mathbf f_3$ so that $b$ cannot transfer valence to $a$, regardless of the fact that it received an extra edge pointing toward $a$. The unique good collection of extra edges for $\Sigma$ is in fact the one that extends it to $\Gamma_4$. With respect to $\mathbf f_4$, the hyperedge $d$ is internally inactive and its favorite is $a$, but since the set $\{\,a,d\,\}$ is still tight, the favorite of the other internally inactive hyperedge, $c$, ends up being $b$. In the case of $\Gamma_4$ the extra edges are received exactly by $d$ and $c$ and they indeed point toward their respective favorites.

This same example is a good illustration of the first (well-definedness) part of the proof, too. If we think of the black line segments of Figure \ref{fig:badcase} as connections between the appropriate emerald markers $\mathbf f_i^+$, then in that sense they form a proper regular square. Each side is in fact a feeler (in this case neither crosses into more than two maximal simplices). For example with respect to $\mathbf f_1$, the hyperedge $c$ is internally inactive and its favorite is $b$. Since a transfer of valence at $\mathbf f_1$ from $c$ to $b$ results in $\mathbf f_2$, the corresponding feeler connects $\mathbf f_1^+$ to $\mathbf f_2^+$. In $S_E$ (cf.\ \eqref{eq:cross}), the $a$-coordinates of the emerald markers dilated from $\mathbf f_1^+$ and $\mathbf f_2^+$ are both $\frac15\left(2+\frac14\right)=0.45$, and the $a$-coordinates of the points corresponding to $\mathbf f_3^+$ and $\mathbf f_4^+$ are both $\frac15\left(1+\frac14\right)=0.25$, cf.\ \eqref{eq:mark} and \eqref{eq:smallmarker}, whereas any point along $\sigma\cap S_E$ has $a$-coordinate $\frac15+\frac15=0.4$ ($\sigma$ is parametrized by weights on the edges of $\Sigma$ and along $S_E$, the sum of the weights on edges adjacent to the same violet vertex has to be $1/5$). This and a similar computation of $b$-coordinates show that $\sigma$ intersects the square formed by the four feelers, viewed as a subset of $S_E\subset Q_G$, at the midpoint of the southeast quarter square. (The intersection point is interior to $\sigma$ because its barycentric coordinates are $1/10$ on the two edges connecting $c$ and $d$ and $1/5$ on the other four edges.) In particular, as the triangle spanned by the two feelers adjacent to $\mathbf f_4^+$ is disjoint from $\sigma$, we need $\mathbf f_2^+$ (i.e., an application of Lemma \ref{lem:cube}) as well to demonstrate that $\sigma$ is an interior simplex.
\end{ex}


\begin{ex}\label{saal}
Let $G$ be the complete bipartite graph with color classes of size $m+1$ and $n+1$, respectively. It is not hard to verify (see \cite[Example 7.2]{hypertutte}) that in the interior polynomial $I(\xi)$ of either hypergraph that $G$ induces, the coefficient of $\xi^k$ is ${m \choose k}{\,n\, \choose k}$. On the other hand, the root polytope $Q_G$ is the product of an $m$- 
and an $n$-dimensional unit simplex, so that the number of lattice points in $s \cdot Q_G$ is ${s+m \choose m}{s+n \choose n}$ 
and thus the Saalsch\"utz formula \eqref{saalschutz} follows from \eqref{eq:ehrhart} and \eqref{eq:goal}. 

It is in fact possible to give a proof of the classical Saalsch\"utz formula using the point of view of this paper but without relying on the interior polynomial and Theorem \ref{thm:main}. Let us consider the following concrete triangulation of $Q_G$ \cite{crossingless}. We denote the color classes of $G$ by $E=\{\,e_0,e_1,\ldots,e_m\,\}$ and $V=\{\,v_0,v_1,\ldots,v_n\,\}$. We fix two horizontal lines on the plane and write the symbols for the elements of $E$ and $V$, respectively, on them from left to right in the indicated order. We call a subgraph of $G$ \emph{non-crossing} if the straight line segments in our figure that represent its edges only intersect at endpoints. The maximal simplices that correspond to two non-crossing spanning trees share a common facet if and only if the two trees differ by a single (obviously defined) `$\text{N}\leftrightarrow\text{\reflectbox{N}}$ transition.' A non-crossing spanning tree can be uniquely described by a \emph{zigzag}, that is a non-crossing path in $G$ containing the leftmost edge $e_0v_0$ and the rightmost edge $e_mv_n$. Here the zigzag is a subgraph of the corresponding tree and the degree $2$ vertices of the zigzag are exactly the $\text{degree}\ge2$ vertices of the tree.

It turns out \cite{crossingless} that the collection of simplices in $Q_G=\Delta_E\times\Delta_V$ that correspond to non-crossing spanning trees is a shellable triangulation. For example, it is not hard to verify that the hypertree order \eqref{eq:order} induces a shelling order. 
It is also not hard to show that in order for a simplex to have $k$ adjacent (through a facet) simplices that are smaller in the shelling order, its zigzag has to have exactly $k$ degree $2$ vertices among $e_1,\ldots,e_m$. Once we have fixed those, the zigzag will be uniquely determined by the choice of its $k$ degree $2$ vertices among $v_0,\ldots,v_{n-1}$. Hence we see that there are ${m \choose k}{n \choose k}$ such trees and therefore \eqref{saalschutz} follows by the argument outlined in Remark \ref{rem:shell}.

The version \eqref{otherform} of Saalsch\"utz's identity can be derived in essentially the same way, by counting interior lattice points in $(q+1)Q_G$ with the help of the same triangulation as above, using the process explained at the beginning of this section.
\end{ex}

The following obvious corollary of Theorem \ref{thm:main} was stated as Conjecture 7.1 in \cite{hypertutte} and was also mentioned in \cite{jkr,homfly}.

\begin{cor}\label{cor:duality}
Any pair $\mathscr H,\overline{\mathscr H}$ 
of abstract dual (transpose) hypergraphs satisfies \[I_{\mathscr H}(x)=I_{\overline{\mathscr H}}(x).\]
\end{cor}

Now that we know that the interior polynomial is not just an invariant of a hypergraph but actually an invariant of the underlying bipartite graph, it becomes an interesting problem to express individual coefficients directly in terms of the graph (i.e., without either breaking the symmetry between the color classes or using the root polytope). As we recalled (and re-proved) in Remark \ref{rem:coefficients}, such formulas are known \cite{hypertutte} for the constant term (which is always $1$) and the linear coefficient (which is the first Betti number). Here we present one for the quadratic coefficient. 

\begin{prop}\label{pro:a_2}
Let the connected bipartite graph $G$ have first Betti number (nullity) $b_1$ and let it have $N$ cycles of length four. Then the quadratic coefficient $a_2$ in the common interior polynomial of the hypergraphs induced by $G$ is ${b_1+1 \choose 2}-N$.
\end{prop}

\begin{proof}
Let the color classes of $G$ be $E$ and $V$. The formula \eqref{eq:coefficients} gives $a_2={d+1 \choose 2}-df_0+f_1$, where $d=|E|+|V|-2$ is the dimension of the root polytope $Q_G$, $f_0$ is the number of edges in $G$, and $f_1$ is the number of $1$-dimensional simplices in an arbitrary triangulation of $Q_G$. Since $b_1=f_0-d-1$, it suffices to show that $N={f_0 \choose 2}-f_1$. In other words, we wish to prove that in any triangulation $\mathscr T$ of $Q_G$, the set of pairs of vertices that are not connected by an edge in $\mathscr T$ is in a one-to-one correspondence with the set of four-cycles of $G$.



Any pair $\varepsilon,\delta$ of edges of $G$ is disjoint from $|E|+|V|-4$ (or $|E|+|V|-3$, if $\varepsilon$ and $\delta$ share an endpoint) star-cuts. The supporting hyperplanes of these cuts (as in Lemma \ref{lem:support}, but this time let us take their intersections with the affine hull of $Q_G$) intersect in an affine subspace\footnote{For some very small graphs we may be talking about an intersection of zero hyperplanes, which we take to be the affine hull of $Q_G$.}, which in turn intersects $Q_G$ in the convex hull of $2$, $3$, or $4$ vertices. Indeed, if $\varepsilon$ and $\delta$ are adjacent then there are no other edges connecting their endpoints. Otherwise, there may be one or two such additional edges, the latter case being when $\varepsilon$ and $\delta$ are opposite edges along a four-cycle. (As $G$ is bipartite and does not have multiple edges, there may be at most one such four-cycle.) Therefore the segment in $Q_G$ determined by $\varepsilon$ and $\delta$ is either an edge of $Q_G$ and hence part of $\mathscr T$, or a diagonal of a quadrilateral (in fact, square) face of $Q_G$. Since $\mathscr T$ induces triangulations of all faces, exactly one of the two diagonals will be a one-simplex in $\mathscr T$. Finally, square faces of $Q_G$ are in a bijection with four-cycles of $G$: we saw how a four-cycle gives rise to a square face and conversely, the four vertices of a square are affinely dependent, i.e., the four corresponding edges form a cycle.
\end{proof}

\begin{ex}
Regarding our two running examples, $K_{2,3}$ has nullity $b_1=2$ with $3$ four-cycles, implying $a_2=0$. The graph of Figure \ref{fig:12a1097} also has $N=3$ but with $b_1=4$, yielding $a_2=7$. These match the results of Examples \ref{ex:k23poly} and \ref{ex:12a1097}.
\end{ex}

Corollary \ref{cor:duality} implies a new formula for $T_G(x,1)$ of an ordinary (not necessarily bipartite) graph $G=(V,E)$. This can be used, for instance, to write the so-called reliability polynomial of an arbitrary connected plane graph as a generating function of activities associated to its regions. To set up the formula, enumerate all possible valence distributions $\mathbf h\colon V\to\N$ of spanning trees in the graph $G'$ that is obtained from $G$ by adding a new vertex in the middle of each edge. (Say that we subtract $1$ from the actual valence at each $v\in V$, although this is not important right now.) All spanning trees of $G'$ are obtained from spanning trees of $G$ by adding one of the two halves for each of the $b_1(G)$ external edges. Despite this multitude of choices, the set $B_{(E,V)}$ of our valence distributions ends up being equinumerous with the set $B_{(V,E)}$ of spanning trees.

\begin{cor}\label{cor:reliable}
Order the set $V$ arbitrarily and compute the internal activity $\iota(\mathbf h)$ for each $\mathbf h\in B_{(E,V)}$ as in Definition \ref{def:activity}. Then we have $T_G(x,1)=\sum_{\mathbf h}x^{\iota(\mathbf h)-1}$.
\end{cor}

In spite of the fact that the case of a plane bipartite graph is often simpler than the general one, due mainly to the presence of the dual graph of the planar embedding which is naturally directed and balanced, the following two corollaries eluded proof until now. (Plane bipartite graphs also form triples called trinities. See the seminal paper \cite{ttt} by Tutte that contains the proof of the so-called Tree Trinity Theorem, or \cite[Sections 8-10]{hypertutte}.) The first one was predicted in \cite[Subsection~10.3]{hypertutte}, where it is also explained why the planar duality formula \cite[Theorem 8.3]{hypertutte} and Corollary \ref{cor:duality} suffice to establish it.

\begin{cor}\label{cor:polynomials}
The interior and exterior polynomials associated to the six hypergraphs induced by the  plane bipartite graphs in the same trinity altogether form a three-element set.
\end{cor}

Shapiro and the second author defined so-called parking functions for an arbitrary directed graph \cite{ps}. These objects have a natural enumerator $p$ which is a one-variable polynomial with positive integer coefficients. Such a parking function enumerator can in particular be associated to the dual graph of a plane bipartite graph (see, e.g., \cite{homfly} for details). Combining Theorem \ref{thm:main} with \cite[Corollary~1.5]{homfly} yields our last result. In the case when one of the hypergraphs induced by the bipartite graph is in fact a graph, this easily follows from a formula of Merino L\'opez \cite{merino} via the connection noted in \cite{ps} between parking functions and the chip firing game (a.k.a.\ abelian sandpile model). The general case is however new.

\begin{cor}\label{cor:park}
Let $G$ be a connected plane bipartite graph so that the common interior polynomial of its induced hypergraphs is $I$. Then $I=p$, where $p$ is the parking function enumerator associated to the dual graph $G^*$.
\end{cor}

\end{document}